\documentclass[11pt,english]{amsart}
\usepackage[T1]{fontenc}
\usepackage[latin9]{inputenc}
\usepackage{mathrsfs}
\usepackage{amsthm}
\usepackage{amssymb}

\makeatletter
\numberwithin{equation}{section}
\numberwithin{figure}{section}
\theoremstyle{plain}
\newtheorem{thm}{\protect\theoremname}
\theoremstyle{definition}
\newtheorem{defn}[thm]{\protect\definitionname}
\theoremstyle{plain}
\newtheorem{prop}[thm]{\protect\propositionname}
\theoremstyle{plain}
\newtheorem{lem}[thm]{\protect\lemmaname}
\theoremstyle{remark}
\newtheorem{rem}[thm]{\protect\remarkname}
\theoremstyle{plain}
\newtheorem{cor}[thm]{\protect\corollaryname}

\makeatother

\usepackage{babel}
\providecommand{\corollaryname}{Corollary}
\providecommand{\definitionname}{Definition}
\providecommand{\lemmaname}{Lemma}
\providecommand{\propositionname}{Proposition}
\providecommand{\remarkname}{Remark}
\providecommand{\theoremname}{Theorem}

\begin{document}

\title[Differential Equations Driven by Variable Order Hölder Noise]{Differential Equations Driven by Variable Order Hölder Noise, and
the Regularizing Effect of Delay }

\author{Fabian A. Harang \\
University of Oslo}

\address{Fabian A. Harang\\
Department of Mathematics\\
 University of Oslo\\
 Email: fabianah@math.uio.no}


\begin{abstract}
In this article we extend the framework of rough paths to processes
of variable Hölder exponent or $variable$ $order$ $paths$. A typical
example of such paths is the multifractional Brownian motion, where
the Hurst parameter of a fractional Brownian motion is considered
to be a function of time. We show how a class of multiple discrete
delay differential equations driven by signals of variable order are
especially well suited to be studied in a path wise sense. In fact,
under some assumptions on the Hölder regularity functions of the driving
signal, the local Hölder regularity of the variable order signal may
be close to zero, without more than two times differentiable diffusion
coefficients (in contrast to higher regularity requirements known
from constant order rough path theory). Furthermore, we give a canonical
algorithm to construct the iterated integral of variable order on
the domain unit square, by constructing the iterated integrals on
some well chosen strip around the diagonal of unit square, and then
extending it to the whole domain using Chens relation. We show how
we can apply this to construct the rough path for a multifractional
Brownian motion. At last, we give a recollection of well known results
on differential equations from the theory rough path, but tailored
for variable order paths. 
\end{abstract}

\maketitle
\tableofcontents{}

\section{Introduction }

In the last 25 years, the field of stochastic analysis has been enriched
with the deterministic tools of what is called Rough Path theory.
This recent theory has proven powerful in fields ranging from the
pure analysis of differential equations to the applied machine learning.
Using only the path properties of irregular signals, the theory are
able to study stochastic processes from a purely deterministic point
of view, making this theory especially interesting for applied mathematicians
or physicists with interest in phenomena of irregular behavior without
formal background in probability theory. 

A couple of years before T. Lyons published his seminal paper \cite{TLyons}
introducing the concept of rough paths, Levy-Vèhel and Peltier published
a paper defining a new stochastic processes called multifractional
Brownian motion, a generalization of the well known fractional Brownian
motion allowing the Hurst exponent to be a function of time (see \cite{Peltier})
. This processes quickly gained interest in applications as it could
describe variations in the local regularity of many phenomena such
as internet traffic, geophysics, image synthesis, and financial markets
(see f.ex, \cite{LebovitsLevy}, \cite{Bianchi}). The most interesting
aspect of the processes from the current article's perspective is
the fact that the (point-wise) Hölder regularity of the multifractional
Brownian motion varies in time according to a regularity function.
That is, if we consider a continuous function $h:\left[0,T\right]\rightarrow\left(0,1\right)$,
the multifractional Brownian motion $t\mapsto B_{t}^{h}$ has the
property that for some $\xi>0$, 
\[
|B_{t+\xi}^{h}-B_{t}^{h}|\lesssim|\xi|^{h(t)}\,\,\,\forall t\in\left[0,T-\xi\right].
\]
 We therefore say that $t\mapsto B_{t}^{h}$ has a local Hölder property.
Of course, restricting our self to $C^{1}$ functions $h:\left[0,T\right]\rightarrow\left[h_{*},h^{*}\right]\subset\left(0,1\right)$
one can show that $t\mapsto B_{t}^{h}$ is globally Hölder continuous
(in the classical sense) with Hölder coefficient $h_{*}$. The multifractional
Brownian motion introduced in \cite{Peltier} by Peltier and Levy-Vehel,
is defined as a generalized fBm on Mandelbrot-Van Ness form, in the
sense that 
\[
B_{t}^{h}:=\int_{-\infty}^{0}\left(t-r\right)^{h\left(t\right)-\frac{1}{2}}-\left(-r\right)^{h\left(t\right)-\frac{1}{2}}dB_{r}+\int_{0}^{t}\left(t-r\right)^{h\left(t\right)-\frac{1}{2}}dB_{r}.
\]
 Such processes ( and other types of processes of local Hölder regularity)
has received much study from the probabilistic and modeling point
of view, see for example \cite{AyCoVe,BenJafRed,Bianchi2013,Bianchi},
but has not yet been extensively studied from a stochastic analysis
point of view with respect to stochastic integration and differential
equations. To the best of our knowledge, only three articles studies
integration and differential equations driven by multifractional noise,
two of which is giving a construction of the stochastic integral through
white noise theory \cite{LebovitsLevy} and through tangent fractional
Brownian motions \cite{LebVehHer} and the third considers differential
equations with additive multifractional noise through an application
of a generalized Girsanov theorem \cite{Harang}. One of the goals
of the current article is to give a path-wise treatment of differential
equations driven by processes of variable local Hölder regularity,
using the framework of rough paths developed in \cite{FriHai} and
combining with the notion of variable order Hölder spaces introduced
by Samko and Ross in \cite{SamRos}. We believe that the rough Path
framework is more suited to study paths of variable local regularity,
not requiring martingale property, or other probabilistic properties,
but fully exploits the regularity properties of the process (which
we believe is one of the main points of interest of this type of process).
However, the rough path theory is centered around (globally) Hölder
continuous functions in the classical sense. We will therefore show
in this article a way to consider variable order paths in a rough
path setting without considering the minimum regularity of the regularity
function $\alpha$. Moreover, we will show how to construct iterated
integrals of multifractional Brownian motion (in a rough path sense)
satisfying the Chen relation and regularity requirements. We will
use processes in variable order Hölder spaces to study delay equations
driven by such noise. In fact, we show that an appropriate choice
of delay times in accordance with the Hölder regularity function of
a process gives existence and uniqueness of delay equations in a variable
order Young type setting. Let us motivate this by the following: 

Classical Young integration lets us define integrals of the form $\int YdX$
for (globally) Hölder continuous paths $X\in\mathcal{C}^{\alpha}$
and $Y\in\mathcal{C}^{\beta}$ with $\alpha+\beta>1.$ Similarly,
if $X$ and $Y$ are locally Hölder continuous in the sense that for
$\alpha,\beta\in C^{1}\left(\left[0,T\right];\left(0,1\right)\right)$
and small $\xi>0$ we have 
\[
|X_{t+\xi}-X_{t}|\lesssim|\xi|^{\alpha(t)},|Y_{t+\xi}-Y_{t}|\lesssim|\xi|^{\beta(t)},
\]
 we would require in Young theory that 
\begin{equation}
\left(\inf_{t\in\left[0,T\right]}\alpha\left(t\right)\right)+\left(\inf_{t\in\left[0,T\right]}\beta\left(t\right)\right)>1\label{eq:classic Young requirement}
\end{equation}
to construct the Young integral $\int YdX$ in the classical sense.
Consider now a dissection of $\left[0,T\right]$ given by $\left\{ \rho_{i}\right\} _{i=1}^{n}$,
i.e 
\[
\left[0,T\right]=\bigcup_{i=1}^{n-1}\left[\rho_{i},\rho_{i+1}\right]
\]
, by standard linearity properties of the integral, we can write 
\[
\int_{0}^{T}YdX=\sum_{i=1}^{n-1}\int_{\rho_{i}}^{\rho_{i+1}}YdX.
\]
Each integral in the sum on the right hand side above can be defined
locally, and we can see that the regularity $X$ is complementing
the regularity of $Y$ locally (on $\left[\rho_{i},\rho_{i+1}\right]\subset\left[0,T\right]$
), rather than globally. We are now tempted to take the limit when
$n\rightarrow\infty$, however, we must be careful to check that it
converges, and we must ensure that we can actually ``glue'' together
this sum uniquely in the limit when $\rho_{i}\rightarrow\rho_{i+1}$
for all $i$, but intuitively we can see that the construction of
the left hand side should have a weaker requirement than Equation
(\ref{eq:classic Young requirement}). In fact, it seems way more
appropriate to use a local complimentary Young requirement, which
we will later prove to be 
\[
\inf_{t\in\left[0,T\right]}\left(\alpha\left(t\right)+\beta\left(t\right)\right)>1.
\]
 It is clear that this requirement is much weaker, as this essentially
translates to the fact that the local regularity of $X$ must be complementary
to the local regularity of $Y$ at each time $t\in\left[0,T\right]$.
With the construction of pathwise variable order integration, we will
then study delayed differential equations driven by variable order
noise of the form 
\begin{equation}
Y_{t}=\xi_{0}+\int_{0}^{t}f\left(Y_{t-\tau_{1}},Y_{t-\tau_{2}},..,Y_{t-\tau_{k}}\right)dX_{t}\,\,\,Y_{u}=\xi_{u}\,\,for\,\,\,u\in\left[-\tau_{k},0\right],\label{eq:multiple delay equation}
\end{equation}
where $X\in\mathcal{C}^{\alpha\left(\cdot\right)}\left(\left[0,T\right];\mathbb{R}^{d}\right)$
and $f\in C_{b}^{2}\left(\mathbb{R}^{n,k};\mathbb{R}^{d\times n}\right)$,
and $\left\{ \tau_{i}\right\} _{i=1}^{k}$ is a sequence of delay
times chosen carefully. This type of equation was studied in a rough
path framework by Neuenkirch, Nourdin and Tindel in \cite{NeuenNourTind},
where the authors proved existence and uniqueness of the above equation
when $X\in\mathcal{C}^{\alpha},$with $\alpha>\frac{1}{3}$ being
constant. In this article, we will show that unique solutions exists
to the above equation when the driving noise has local $\alpha-$Hölder
regularity which might be very close to $0$, and the diffusion coefficient
$f$ is only in $C_{b}^{2}$, when the delay is chosen appropriately.
Hence this improvement of regularity comes at the cost of restricting
the relationship between regularity functions and delay times, and
thus restricting the flexibility of the equation. In the variable
order framework, the existence and uniqueness of this equation depends
on the relationship between the delay times and the regularity function
in the sense that we must require 
\[
\inf_{t\in\left[\tau,T\right];i\in\left\{ 1,..,k\right\} }\left\{ \alpha\left(t\right)+\alpha\left(t-\tau_{i}\right)\right\} >1
\]
to obtain existence and uniqueness in a Young-type regime (i.e no
need for iterated integrals). The above requirement does not give
any lower bounds on $\alpha:\left[0,T\right]\rightarrow\left(0,1\right)$,
and hence opens up to study delay equations for paths of locally extreme
irregularity, as long as they are compensated through appropriately
chosen delay sequence $\left\{ \tau_{i}\right\} $. We can therefore
think of the chosen delay times as regularizing the solution in terms
of requirements on the diffusion coefficient. If one choose the regularity
function to be a sine wave of the form $\alpha\left(t\right)=\beta+\epsilon\sin\left(t\right)$
for some $\epsilon$ and $\beta$ such that $\alpha$ takes values
in $\left(0,1\right)$, we can see that for some appropriately chosen
constants, there exists at least one $\tau\in\left[0,T\right]$ such
that 
\[
\alpha\left(t\right)+\alpha\left(t-\tau\right)>1.
\]
 It is not difficult to come up with several more examples of regularity
functions which satisfy such a requirement. 

The article is structured as follows: 
\begin{itemize}
\item Section \ref{sec:framework} gives an introduction and some basic
concepts of variable order Hölder spaces and results relating to the
sewing lemma for variable order paths. 
\item In section \ref{sec:yOUNG-delay-equations} we show existence of a
unique solution to a multiple discrete delay equation driven by variable
order noise on the form of Equation (\ref{eq:multiple delay equation})
under some periodic type of regularity requirement on the regularity
function of a variable order path. 
\item In section \ref{sec:Variable-order-rough path and mBm} we give a
definition of what we mean by a variable order rough path, and shows
how to construct the rough path with respect to a multifractional
Brownian motion. For general variable order noise we provide an algorithm
to construct iterated integrals provided that the regularity function
is sufficiently smooth and one can construct an iterated integral
locally on a small sub square of $\left[0,T\right]$. 
\item Section \ref{sec:Controlled paths and variabel order RDE} is devoted
to showing how the rough path framework from \cite{FriHai} may be
used for variable order paths. This section is rather meant for those
not familiar with Rough Paths theory from before, as most of the results
here are simple generalizations of what can be found in \cite{FriHai}.
It is included however for to ensure that the article is self contained
. 
\item At last we have an appendix for an embedding result showing that if
a generalized Gagliardo norm is bounded, then we have variable order
Hölder continuity of functions. 
\end{itemize}

\subsection{Notation}

We will give a short recollection on some common notation in this
article: For a function $f:\mathbb{R}_{+}\rightarrow\left[a,b\right]$
We frequently write $f_{*}=a$ and $f^{*}=b$. We will use the constant
$C$ to denote a general constant when considering inequalities. This
constant may change throughout calculations, but when it is important
to give dependence of different variables, we write that variable
in subscript, i.e $C_{f}$ for $C(f)$. Otherwise we mostly adopt
the notation used in \cite{FriHai}, especially the convention that
$f_{s,t}=f(t)-f(s)$ is used to denote increments. 

\section{\label{sec:framework}Variable Order Hölder spaces and Sewing Lemma}

In this section we will give formal definitions of what we mean with
variable order rough paths. We will present some properties of variable
order Hölder spaces first introduced by S. Samko and B. Ross in \cite{SamRos},
and later studied in . Let us first give a definition of such a space.
\begin{defn}
Let $\alpha\in C\left(\left[0,T\right];\left(0,1\right)\right)$ satisfying
$\sup_{t\in\left[0,T\right];|h|\leq1}|h|^{\alpha\left(t\right)-\alpha\left(t+h\right)}\leq C$,
and consider a function $f:\left[0,T\right]\rightarrow\mathbb{R}^{n}$.
We define the Hölder space of variable order $\mathcal{C}^{\alpha\left(\cdot\right)}\left(\left[0,T\right];\mathbb{R}^{n}\right)$
to consist of functions $f$ such that 
\[
\parallel f\parallel_{\alpha\left(\cdot\right);\left[0,T\right]}:=\sup_{0\leq t\leq T;|h|\leq1}\frac{|f_{t,t+h}|}{|h|^{\alpha\left(t\right)}}<\infty.
\]
\end{defn}

Moreover, the mapping $f\mapsto|f\left(0\right)|+\parallel f\parallel_{\alpha\left(\cdot\right);\left[0,T\right]}$
induces a Banach space structure on $\mathcal{C}^{\alpha\left(\cdot\right)}\left(\left[0,T\right];\mathbb{R}^{n}\right)$.
Moreover, denote by $\mathcal{C}_{2}^{\beta\left(\cdot\right)}\left(\left[0,T\right]^{2};\mathbb{R}^{n}\right)$
the space of all two-variable functions $g$ satisfying 
\[
\parallel g\parallel_{\beta\left(\cdot\right);\left[0,T\right]}:=\sup_{0\leq t\leq T;|h|\leq1}\frac{|g_{t,t+h}|}{|h|^{\beta\left(t\right)}}<\infty,
\]
where in this case $g_{t,t+h}=g(t,t+h).$ 
\begin{description}
\item [{A1}] We will throughout this article assume that the regularity
functions $\alpha:\left[0,T\right]\rightarrow\left(0,1\right)$ we
consider satisfy the following requirement 
\[
\sup_{t\in\left[0,T\right];|h|\leq1}|h|^{\alpha\left(t\right)-\alpha\left(t+h\right)}\leq C
\]
. Sometimes this will also be clearly specified in the lemmas or theorems
to clarify the importance of this, but often we assume this condition
without further notice. 
\end{description}
There are several equivalent definitions of this space, and therefore,
we will give a lemma showing equivalence of different definitions
of variable order Hölder semi-norms. 
\begin{prop}
\label{prop:Equivalence of Variable H=0000F6lder norm}Let $\alpha\in C\left(\left[0,T\right];\left(0,1\right)\right)$
be a regularity function such that 
\[
\sup_{t\in\left[0,T\right];|h|\leq1}|h|^{\alpha\left(t\right)-\alpha\left(t+h\right)}\leq C,
\]
 we have that the following three expressions forms equivalent norms
on $\mathcal{C}^{\alpha\left(\cdot\right)}$
\[
\sup_{0\leq t\leq T;|h|\leq1}\frac{|f_{t,t+h}|}{|h|^{\alpha\left(t\right)}}\sim\sup_{0\leq t\leq T;|h|\leq1}\frac{|f_{t,t+h}|}{|h|^{\alpha\left(t+h\right)}}
\]
\[
\sim\sup_{0\leq s<t\leq T;|t-s|\leq1}\frac{|f_{s,t}|}{|t-s|^{\max\left(\alpha\left(t\right),\alpha\left(s\right)\right)}},
\]
where $\sim$ denotes equivalence. 
\end{prop}

\begin{proof}
The first equivalence is a simple consequence of the continuity assumption
on $\alpha$, i.e consider for some very small $\epsilon$ it follows
that 
\[
\frac{|f_{t,t+h}|}{|h|^{\alpha\left(t+h\right)}}=\frac{|f_{t,t+h}|}{|h|^{\alpha\left(t\right)}}|h|^{\alpha\left(t\right)-\alpha\left(t+h\right)},
\]
and the equivalence follows taking supremum on both sides and using
\[
\sup_{t\in\left[0,T\right];|h|\leq1}|h|^{\alpha\left(t\right)-\alpha\left(t+h\right)}\leq C.
\]
 The equivalence 
\[
\sup_{0\leq t\leq T;|h|\leq1}\frac{|f_{t,t+h}|}{|h|^{\alpha\left(t+h\right)}}\sim\sup_{0\leq s<t\leq T;|t-s|\leq1}\frac{|f_{t,t+h}|}{|t-s|^{\max\left(\alpha\left(t\right),\alpha\left(s\right)\right)}}
\]
 follows in the same way as the previous argument. 
\end{proof}
We will need an important scaling property of variable order Hölder
norms, which is a simple generalization of Exercise 4.24in \cite{FriHai}. 
\begin{lem}
\label{lem: Scaling of H=0000F6lder norms}Assume $g\in\mathcal{C}^{\alpha\left(\cdot\right)}\left(\left[s,t\right]\right)$
for some regularity function $\alpha:\left[0,T\right]\rightarrow\left(0,1\right)$
and for any $s<t\in\left[0,T\right]$ where $0<|t-s|\leq h<1$ such
that there exists an $M>0$ with 
\[
\sup_{0<s<t<T;|t-s|\leq h}\frac{|g_{s,t}|}{|t-s|^{\alpha\left(s\right)}}\leq M.
\]
 Then $g\in\mathcal{C}^{\alpha\left(\cdot\right)}\left(\left[0,T\right]\right)$
with 
\[
\parallel g\parallel_{\alpha\left(\cdot\right);\left[0,T\right]}\leq M\left(1\vee2h^{-1}\right).
\]
\end{lem}

\begin{proof}
Without loss of generality, we may consider $M=1$, and write $\alpha_{i}=\inf\left\{ \alpha\left(t\right);t\in\left[t_{i-1},t_{i}\right]\right\} $.
For fixed $0\leq s<t\leq T$ with $|t-s|\geq h$, divide $\left[s,t\right]$
into $N=\inf\left\{ k\in\mathbb{N};\frac{|t-s|}{h}\leq k\right\} $
intervals by setting $t_{i}=\left(s+ih\right)\wedge t$ for $i=1,..,N$
. Then clearly $|t_{i}-t_{i-1}|\leq h$, and $t_{N}=t$ . We can then
write 
\[
|g_{s,t}|\leq\sum_{i=1}^{N}|g_{t_{i-1},t_{i}}|\leq M\sum_{i=1}^{N}|h|^{\alpha_{i}}\leq M\left(1+\frac{|t-s|}{h}\right),
\]
since $N\leq1+\frac{|t-s|}{h}$. Furthermore, $M\left(1+\frac{|t-s|}{h}\right)=Mh^{-1}\left(h+|t-s|\right)\leq M2h^{-1}|t-s|$.
Taking supreme over $s,t\in\left[0,T\right]$, we get 
\[
\parallel g\parallel_{\alpha\left(\cdot\right);\left[0,T\right]}=\sup_{0\leq s<t\leq T}\frac{|g_{s,t}|}{|t-s|^{\alpha\left(s\right)}}\leq M\left(1\vee2h^{-1}\right).
\]
\end{proof}

\subsection{Sewing lemma for variable order paths}

The sewing lemma is a crucial piece of the theory of rough paths,
and gives us a construction of abstract integration maps. Although
the generalization to variable order Hölder paths is a quite simple
procedure, it is important for variable order paths that the integration
depend on the local regularity of the integrand and integrator, rather
than global regularity as in standard rough path theory. Similarly
with the notation and definitions found in \cite{FriHai}, we introduce
the space $\mathcal{C}_{2}^{\alpha\left(\cdot\right),\beta\left(\cdot\right)}\left(\left[0,T\right];\mathbb{R}^{d}\right)$
being defined by the norm 
\[
\parallel f\parallel_{\alpha\left(\cdot\right),\beta\left(\cdot\right);\left[0,T\right]}:=\sup_{u\in\left(s,s+h\right);|h|\leq1}\frac{|f_{s,s+h}|}{|h|^{\alpha\left(s\right)}}+\sup_{s<u<t\in\left[0,T\right]}\frac{|\left(\delta f\right)_{s,u,t}|}{|t-s|^{\beta\left(s\right)}}
\]
\[
=:\parallel f\parallel_{\alpha\left(\cdot\right);\left[0,T\right]}+\parallel\delta f\parallel_{\beta\left(\cdot\right);\left[0,T\right]},
\]
where $\alpha:\left[0,T\right]\rightarrow\left(0,1\right)$ is a $C^{1}$
function, and $\beta:\left[0,T\right]\rightarrow[\beta_{*},\infty)\subset(1,\infty)$
and the operator $\delta$ is such that 
\[
\left(\delta f\right)_{s,u,t}=f_{s,t}-f_{s,u}-f_{u,t}.
\]
 With this definition, we recite the sewing lemma customized to variable
order integrands. 
\begin{lem}
\label{lem:-SEWING LEMMA}$\left(Sewing\,\,\,Lemma\right)$ Let $\alpha:\left[0,T\right]\rightarrow\left(0,1\right)$
and $\beta:\left[0,T\right]\rightarrow[\beta_{*},\infty)\subset(1,\infty)$
be two $C^{1}$ functions. Then there exists a unique continuous map
\[
\mathcal{I}:\mathcal{C}_{2}^{\alpha\left(\cdot\right),\beta}\left(\left[0,T\right];\mathbb{R}^{d}\right)\rightarrow\mathcal{C}^{\alpha\left(\cdot\right)}\left(\left[0,T\right];\mathbb{R}^{d}\right)
\]
 with $\left(\mathcal{I}\varXi\right)_{0}$ and satisfies 
\[
|\left(\mathcal{I}\varXi\right)_{s,t}-\varXi_{s,t}|\leq C_{\beta}\parallel\delta\varXi\parallel_{\beta\left(\cdot\right);\left[0,T\right]}|t-s|^{\beta_{*}},
\]
 where $C$ depends only on $\beta$ and $\parallel\delta\varXi\parallel_{\beta\left(\cdot\right);\left[0,T\right]}$.
\end{lem}

\begin{proof}
We will use the proof of the sewing lemma found in \cite{FriHai}
(Lemma 4.2), and modify this where it is needed. Instead of reciting
all of the steps from this proof, we will only show the steps which
is crucial to us, and encourage the interested reader to confirm the
rest of the proof in \cite{FriHai}. 

Consider a dyadic partition $\mathcal{P}_{n}$ of $\left[s,t\right]$
defined iterativly by setting $\mathcal{P}_{0}=\left\{ \left[s,t\right]\right\} $
and 
\[
\mathcal{P}_{n+1}=\bigcup_{\left[u,v\right]\in\mathcal{P}_{n}}\left\{ \left[u,m\right],\left[m,v\right]\right\} ,
\]
where $m:=\frac{u+v}{2}$. The partition $\mathcal{P}_{n}$ consists
of $2^{n}$ intervals of size $2^{-n}|t-s|$ . Furthermore, define
the integral operator $\left(\mathcal{I}^{0}\varXi\right)_{s,t}=\varXi_{s,t}$,
and iterativly 
\[
\left(\mathcal{I}^{n}\varXi\right)_{s,t}=\sum_{\left[u,v\right]\in\mathcal{P}_{n}}\varXi_{u,v},
\]
and note that 
\[
\left(\mathcal{I}^{n+1}\varXi\right)_{s,t}=\sum_{\left[u,v\right]\in\mathcal{P}_{n+1}}\varXi_{u,v}=\sum_{\left[u,v\right]\in\mathcal{P}_{n}}\varXi_{u,m}+\varXi_{m,v}
\]
\[
=\sum_{\left[u,v\right]\in\mathcal{P}_{n}}\varXi_{u,v}-\sum_{\left[u,v\right]\in\mathcal{P}_{n}}\left(\delta\varXi\right)_{u,m,v}.
\]
and it follows 
\[
|\left(\mathcal{I}^{n+1}\varXi\right)_{s,t}-\left(\mathcal{I}^{n}\varXi\right)_{s,t}|\leq|\sum_{\left[u,v\right]\in\mathcal{P}_{n}}\left(\delta\varXi\right)_{u,m,v}|,
\]
 using variable order Hölder norms it is clear that 
\[
\leq\sum_{\left[u,v\right]\in\mathcal{P}_{n}}\parallel\left(\delta\varXi\right)\parallel_{\beta\left(\cdot\right);\left[u,v\right]}|v-u|^{\beta\left(u\right)}\leq\parallel\left(\delta\varXi\right)\parallel_{\beta\left(\cdot\right);\left[s,t\right]}\sum_{\left[u,v\right]\in\mathcal{P}_{n}}|v-u|^{\beta\left(u\right)}
\]
\[
\leq\parallel\left(\delta\varXi\right)\parallel_{\beta\left(\cdot\right);\left[s,t\right]}2^{-n\left(\beta_{*}-1\right)}|t-s|^{\beta_{*}}
\]
 which imply that the sequence of operators $\mathcal{I}^{n}$ is
Cauchy, and it follows that 
\[
|\left(\mathcal{I}\varXi\right)_{s,t}-\varXi_{s,t}|\leq\sum_{n\geq0}|\left(\mathcal{I}^{n+1}\varXi\right)_{s,t}-\left(\mathcal{I}^{n}\varXi\right)_{s,t}|
\]
\[
\leq C_{\beta_{*}}\parallel\left(\delta\varXi\right)\parallel_{\beta\left(\cdot\right);\left[s,t\right]}|t-s|^{\beta_{*}},
\]
where $C_{\beta_{*}}=\sum_{n\geq0}2^{-n\left(\beta_{*}-1\right)}<\infty$.
Next one needs to check that the above inequality holds also for arbitrary
partitions$\mathcal{P}$ of $\left[s,t\right]$ when $|\mathcal{P}|\rightarrow0$.
This follows from the proof in \cite{FriHai} Lemma 4.2. 
\end{proof}
\begin{rem}
Notice that by the above inequality we have 
\[
|\left(\mathcal{I}\varXi\right)_{s,t}|\lesssim|t-s|^{\alpha\left(s\right)}.
\]
Indeed, the mapping $t\mapsto\left(\mathcal{I}\varXi\right)_{0,t}\in\mathcal{C}^{\alpha\left(\cdot\right)}\left(\left[0,T\right];\mathbb{R}^{d}\right)$
is a consequence of the assumption that $t\mapsto\varXi_{0,t}\in\mathcal{C}^{\alpha\left(\cdot\right)}\left(\left[0,T\right];\mathbb{R}^{d}\right)$
. The abstract integration map is constructed as a Riemann type of
integral of the function $\varXi$, i.e 
\[
\left(\mathcal{I}\varXi\right)_{s,t}:=\lim_{|\mathcal{P}|\rightarrow0}\sum_{\left[u,v\right]\in\mathcal{P}}\varXi_{u,v}.
\]
\end{rem}

\section{Young theory for variable order paths with applications to delay
equations\label{sec:yOUNG-delay-equations}}

The first goal of this section is to show the construction of integrals
of the form 
\[
\int_{0}^{t}Y_{r}dX_{r}
\]
 when $Y\in\mathcal{C}^{\gamma\left(\cdot\right)}$ and $X\in\mathcal{C}^{\alpha\left(\cdot\right)}$
by using the sewing lemma proved in the last section. As discussed
in the introduction, if assumed that $\inf_{t\in\left[0,T\right]}\left\{ \alpha\left(t\right)\right\} +\inf_{t\in\left[0,T\right]}\left\{ \gamma\left(t\right)\right\} >1$,
we could construct Young integrals just as usual. However, as $\alpha$
and $\gamma$ may reach its minimum at completely different times,
the regularization effect of one path on the other is completely ignored.
Essentially we will be forced to only consider $\alpha$ and $\gamma$
such that $\alpha_{*}\in\left(1-\gamma_{*},1\right)$. However, by
combining the regularity functions $\alpha\left(t\right)+\gamma\left(t\right)$
the sum can easily be above $1$ for all $t$ even though the individual
functions may take values between $\left(0,1\right)$. We will show
that when considering Hölder spaces of variable order $\mathcal{C}^{\alpha\left(\cdot\right)}$
and $\mathcal{C}^{\gamma\left(\cdot\right)}$ a criterion for integration
in the Young sense is that 
\[
\beta:=\inf_{t\in\left[0,T\right]}\left\{ \alpha\left(t\right)+\gamma\left(t\right)\right\} >1.
\]
We give the following proposition for Young integration of variable
order
\begin{prop}
\label{prop:Young integrals of variable order}Consider $Y\in\mathcal{C}^{\gamma\left(\cdot\right)}$
and $X\in\mathcal{C}^{\alpha\left(\cdot\right)}$, and define $\varXi_{s,t}=Y_{s}X_{s,t}$,
then 

\[
|\left(\mathcal{I}\varXi\right)_{s,t}-\varXi_{s,t}|\leq C_{\beta}\parallel Y\parallel_{\gamma\left(\cdot\right);\left[s,t\right]}\parallel X\parallel_{\alpha\left(\cdot\right);\left[s,t\right]}|t-s|^{\beta},
\]
where $\beta=\inf_{t\in\left[0,T\right]}\left\{ \alpha\left(t\right)+\gamma\left(t\right)\right\} $. 
\end{prop}

\begin{proof}
First, see that $\left(\delta\varXi\right)_{s,u,t}=-Y_{s,u}X_{u,t}$,
and by standard estimates it follows that 
\[
\parallel\varXi\parallel_{\alpha\left(\cdot\right);\left[0,T\right]}<\infty,
\]
as the regularity of the function is inherited by the driver $X\in\mathcal{C}^{\alpha\left(\cdot\right)}$.
We must check that for $\beta\left(t\right)=\alpha\left(t\right)+\gamma\left(t\right),$
\[
\parallel\delta\varXi\parallel_{\beta\left(\cdot\right);\left[0,T\right]}<\infty.
\]
We know from Lemma \ref{prop:Equivalence of Variable H=0000F6lder norm}
that for $Z\in\mathcal{C}^{h\left(\cdot\right)}$, 
\[
\sup_{s<t\in\left[0,T\right];|t-s|\leq1}\frac{|Z_{s,t}|}{|t-s|^{h(t)}}\sim\sup_{s<t\in\left[0,T\right];|t-s|\leq1}\frac{|Z_{s,t}|}{|t-s|^{h(s)}}\,\,\,for\,\,\,|t-s|\leq1,
\]
and therefore, we can see that 
\[
\sup_{s<u<t}\frac{|\left(\delta\varXi\right)_{s,u,t}|}{|t-s|^{\alpha\left(s\right)+\beta\left(s\right)}}
\]
\[
\leq\parallel Y\parallel_{\gamma\left(\cdot\right)}\parallel X\parallel_{\alpha\left(\cdot\right)}\sup_{s<u<t}\frac{|u-s|^{\gamma\left(u\right)}|t-u|^{\alpha\left(u\right)}}{|t-s|^{\alpha\left(r\right)+\gamma\left(r\right)}}<\infty,
\]
which gives us the desired result in combination with Lemma \ref{lem:-SEWING LEMMA}. 
\end{proof}
With the above result, it is clear that we can define integrals of
Young type for variable order paths, only assuming that the regularity
function of the two paths enjoys complementary Young regularity for
each $t\in\left[0,T\right]$. Next we will use the above integral
construction to show existence and uniqueness of delay equations driven
by variable order paths. Let us first give a definition of a proper
set of regularity functions suitable for the study of delay equations. 
\begin{defn}
\label{def:Regularity function admitting Young delay}Let $\tau=\left(\tau_{1},...,\tau_{k}\right)\in\triangle^{(k)}\left(0,T\right)$,
an let $\Sigma_{\tau}\left(\left[0,T\right];\left(0,1\right)\right)$
denote a subspace of $C^{1}\left(\left[0,T\right];\left(0,1\right)\right)$
defined by all functions $\alpha:\left[0,T\right]\rightarrow\left(0,1\right)$,
satisfying $\mathbf{A1}$ with the property that 
\[
\inf_{t\in\left[\tau,T\right];i\in\left\{ 1,..,k\right\} }\left\{ \alpha\left(t\right)+\alpha\left(t-\tau_{i}\right)\right\} >1.
\]
We say that $\alpha\in\Sigma_{\tau}$ is a regularity function admitting
Young delay. 
\end{defn}

The next Lemma gives us some à-priori bounds of the solution to delayed
Young differential equations driven by variable order noise. 
\begin{lem}
\label{lem:Apriori Bounds on delayed SDE}Let $T=1,$and assume $\alpha$
is the same as in Theorem \ref{thm:Delayed Young RDE}. Consider two
paths $Y\in\mathcal{C}^{\alpha\left(\cdot\right)}\left(\left[0,T\right];\mathbb{R}^{m}\right)$,
and $X\in\mathcal{C}^{\alpha\left(\cdot\right)}\left(\left[0,T\right];\mathbb{R}^{d}\right)$
where $X$ and $Y$ are related through the delayed RDE 
\[
Y_{t}=\xi_{0}+\int_{0}^{t}f\left(Y_{t-\tau_{1}},Y_{t-\tau_{2}},..,Y_{t-\tau_{k}}\right)dX_{t}\,\,\,Y_{u}=\xi_{u}\,\,for\,\,\,u\in\left[-\tau_{k},0\right],
\]
~where $f\in C_{b}^{1}\left(\mathbb{R}^{m};\mathbb{R}^{d\times m}\right)$,
and $\xi\in C^{1}\left(\left[-\tau_{k},0\right];\mathbb{R}^{d}\right)$
and the equation is interpreted in the Young integral sense. Then
we have the bound 
\[
\parallel Y\parallel_{\alpha\left(\cdot\right);\left[0,T\right]}\leq C_{f,\beta}\left(\parallel X\parallel_{\alpha\left(\cdot\right);\left[0,T\right]}\vee\parallel X\parallel_{\alpha\left(\cdot\right);\left[0,T\right]}^{\frac{1}{\gamma}-1}\right)
\]
where 
\[
\gamma:=\left(\inf_{t\in\left[\tau_{k},T\right];i\in\left\{ 1,..,k\right\} }\left\{ \alpha\left(t\right)+\alpha\left(t-\tau_{i}\right)\right\} -\sup_{t\in\left[0,T\right]}\left\{ \alpha\left(t\right)\right\} \right)\in\left(0,1\right).
\]
 
\end{lem}

\begin{proof}
First, do a trivial extension of $\alpha$ such that $\alpha:\left[-\tau_{k},T\right]\rightarrow\left(0,1\right)$
where $\alpha\left(u\right)=1$ when $u\in[-\tau_{k},0)$ and $\alpha(u)=\alpha\left(u\right)$
when $u\in\left[0,T\right]$, this is to account for the $C^{1}$
nature of the initial function $\xi$. We write 
\[
Y_{t}=\xi_{0}+\int_{0}^{t}f\left(Y_{r-\tau_{1}},..,Y_{r-\tau_{k}}\right)dX_{t},
\]
and 
\[
Y_{s,t}=\int_{s}^{t}f\left(Y_{r-\tau_{1}},..,Y_{r-\tau_{k}}\right)dX_{t}
\]
 and remember that $Y_{u}=\xi_{u}$ on $[-\tau,0]$. Using that 
\[
f\left(x_{1},...x_{k}\right)-f\left(y_{1},..,y_{k}\right)
\]
\[
=\sum_{i=1}^{d}f\left(x_{1},.,x_{i},y_{i+1},..y_{k}\right)-f\left(x_{1},.,x_{i-1},y_{i},..y_{k}\right),
\]
we have by the mean value theorem for some $c\in\left(0,1\right)$
\[
|f\left(x_{1},..,x_{k}\right)-f\left(y_{1},..,y_{k}\right)|\leq\sum_{i=1}^{k}|\nabla f\left(x_{1},..,x_{i}c+y_{i}\left(1-c\right),..,x_{k}\right)||x_{i}-y_{i}|,
\]
which combined with proposition \ref{prop:Young integrals of variable order}
gives us the estimate 
\[
|Y_{s,t}-f\left(Y_{s-\tau_{1}},..,Y_{s-\tau_{k}}\right)X_{s,t}|=|\int_{s}^{t}f\left(Y_{r-\tau_{1}},..,Y_{r-\tau_{k}}\right)-f\left(Y_{s-\tau_{1}},..,Y_{s-\tau_{k}}\right)dX_{r}|
\]
\[
\leq C_{\beta,k}\parallel\nabla f\parallel_{\infty}\sup_{i\in\left\{ 1,..,k\right\} }\parallel Y_{\cdot-\tau_{i}}\parallel_{\alpha\left(\cdot-\tau_{i}\right);\left[s,t\right]}\parallel X\parallel_{\alpha\left(\cdot\right);\left[s,t\right]}|t-s|^{\beta},
\]
where $\beta:=\inf_{t\in\left[\tau_{k},T\right];i\in\left\{ 1,..,k\right\} }\left\{ \alpha\left(t\right)+\alpha\left(t-\tau_{i}\right)\right\} $,
and 
\[
\parallel\nabla f\parallel_{\infty}=\sup_{x\in\mathbb{R}^{k}}|\nabla f\left(x_{1},..,x_{i},..x_{k}\right)|.
\]
 Moreover, it is a simple exercise to see that for all $\left[s,t\right]\subset\left[0,T\right]$
\[
\sup_{i\in\left\{ 1,..,k\right\} }\parallel Y_{\cdot-\tau_{i}}\parallel_{\alpha\left(\cdot-\tau_{i}\right);\left[s,t\right]}\leq\parallel Y\parallel_{\alpha\left(\cdot\right);\left[s,t\right]},
\]
and it therefore follows that the regularity of $Y$ is given through
the inequality 
\[
\parallel Y\parallel_{\alpha\left(\cdot\right);\left[s,t\right]}\leq\parallel f\parallel_{\infty}\parallel X\parallel_{\alpha\left(\cdot\right);\left[s,t\right]}
\]
\[
+C_{\beta}\parallel Df\parallel_{\infty}\parallel Y_{\cdot}\parallel_{\alpha\left(\cdot\right);\left[s,t\right]}\parallel X\parallel_{\alpha\left(\cdot\right);\left[s,t\right]}|t-s|^{\beta-\max\left(\alpha\left(t\right),\alpha\left(s\right)\right)}.
\]
We may choose $|t-s|\leq h$ small, such that 
\[
C_{\beta}\parallel f\parallel_{C_{b}^{1}}\parallel X\parallel_{\alpha\left(\cdot\right);\left[0,T\right]}|h|^{\gamma}\leq\frac{1}{2},
\]
where $\gamma=\beta-\sup_{t\in\left[0,T\right]}\left(\alpha\left(t\right)\right)>0$,
and we get that for all $|t-s|\leq h$ 
\[
\parallel Y\parallel_{\alpha\left(\cdot\right);\left[s,t\right]}\leq2\parallel f\parallel_{C_{b}^{1}}\parallel X\parallel_{\alpha\left(\cdot\right);\left[0,T\right]}.
\]
Using Lemma \ref{lem: Scaling of H=0000F6lder norms}, we have that
\[
\parallel Y\parallel_{\alpha\left(\cdot\right);\left[0,T\right]}\leq C_{f,\beta}\parallel X\parallel_{\alpha\left(\cdot\right);\left[0,T\right]}\left(h^{-1}\vee1\right),
\]
but we have that $h\propto\parallel X\parallel_{\alpha\left(\cdot\right);\left[0,T\right]}^{-\frac{1}{\gamma}}$,
which imply 
\[
\parallel Y\parallel_{\alpha\left(\cdot\right);\left[0,T\right]}\leq C_{f,\beta}\left(\parallel X\parallel_{\alpha\left(\cdot\right);\left[0,T\right]}\vee\parallel X\parallel_{\alpha\left(\cdot\right);\left[0,T\right]}^{\frac{1}{\gamma}-1}\right),
\]
where the constant $C_{\beta,f}$ throughout has changed but depend
on $\beta$,$\parallel f\parallel_{C_{b}^{1}}$. 
\end{proof}
Now we are ready to present the main theorem of this section giving
existence and uniqueness results to delayed Young differential equations
driven by variable order noise. 
\begin{thm}
\label{thm:Delayed Young RDE}Assume $\alpha:\left[0,T\right]\rightarrow\left(0,1\right)$
satisfy $\mathbf{A1}$ and is such that for some fixed sequence $\left\{ \tau_{i}\right\} _{i=1}^{k}$
with $0<\tau_{1}<\tau_{2}<..<\tau_{k}$, we have that 
\[
\inf_{t\in\left[\tau_{k},T\right];i\in\left\{ 1,..,k\right\} }\left\{ \alpha\left(t\right)+\alpha\left(t-\tau_{i}\right)\right\} >1.
\]
Assume that $X\in\mathcal{C}^{\alpha\left(\cdot\right)}\left(\left[0,T\right];\mathbb{R}^{d}\right)$
and $f\in C_{b}^{2}\left(\mathbb{R}^{k,n};\mathbb{R}^{d\times n}\right)$.
Then there exists a unique solution to the delayed Young Differential
Equation 
\begin{equation}
dY_{t}=f\left(Y_{t-\tau_{1}},Y_{t-\tau_{2}},..,Y_{t-\tau_{k}}\right)dX_{t},\,\,\,Y_{u}=\xi_{u}\,\,for\,\,\,u\in\left[-\tau_{k},0\right]\label{eq:Delayed RDE}
\end{equation}
 in $\mathcal{C}^{\alpha\left(\cdot\right)}\left(\left[0,T\right];\mathbb{R}^{n}\right)$. 
\end{thm}

Before proving Theorem \ref{thm:Delayed Young RDE}, we will need
another lemma which will be important for the uniqueness of the solution
to the delayed equations, and shows that the difference between functions
acting on variable Hölder continuous paths is again variable Hölder
continuous, and the variable Hölder norm is bounded by the variable
Hölder norm of the input signals. 
\begin{lem}
\label{lem:Regularity of composition of functions}Let $T\leq1$,
let $f\in C_{b}^{1}\left(\mathbb{R}^{k,n};\mathbb{R}^{d\times n}\right)$
with derivative $Df$ of $\theta$-Hölder continuity in each direction,
i.e there exists a $C_{k,\theta}$ such that for each $i\in\left\{ 1,..,k\right\} $
we have $|\nabla f(x_{1},..,x_{i},..x_{k})-\nabla f\left(y_{1},..,y_{i},..,y_{k}\right)|\leq C_{\theta,k}|x_{i}-y_{i}|^{\theta}$
for $\theta\in(0,1]$. Then there exists a constant $M\in\mathbb{R}_{+}$
such that for two paths $Z,\tilde{Z}\in\mathcal{C}^{\alpha\left(\cdot\right)}\left(\left[0,T\right];\mathbb{R}^{d}\right)$
we have $\parallel Z\parallel_{\alpha\left(\cdot\right)}\vee\parallel\tilde{Z}\parallel_{\alpha\left(\cdot\right)}\leq M$.
 Extended $Z_{u}=\tilde{Z}_{u}=\xi_{u}$ for $u\in[-\tau_{k},0)$,
then we have that 
\[
|f\left(Z_{\cdot-\tau_{1}},..,Z_{\cdot-\tau_{k}}\right)_{s,t}-f\left(\tilde{Z}_{\cdot-\tau_{1}},..,\tilde{Z}_{\cdot-\tau_{k}}\right)_{s,t}|
\]
\[
\leq C_{\theta,M,k}\parallel f\parallel_{C_{b}^{1}}\left(|Z_{0}-\tilde{Z}_{0}|+\parallel Z-\tilde{Z}\parallel_{\alpha\left(\cdot\right)}\right)\sum_{i=1}^{k}|t-s|^{\max\left(\alpha\left(t-\tau_{i}\right),\alpha\left(s-\tau_{i}\right)\right)\times\theta}.
\]
 
\end{lem}

\begin{proof}
Look at the difference 
\[
f\left(Z_{\cdot-\tau_{1}},..,Z_{\cdot-\tau_{k}}\right)_{s,t}-f\left(\tilde{Z}_{\cdot-\tau_{1}},..,\tilde{Z}_{\cdot-\tau_{k}}\right)_{s,t}
\]
\[
=f\left(Z_{t-\tau_{1}},..,Z_{t-\tau_{k}}\right)-f\left(\tilde{Z}_{t-\tau_{1}},..,\tilde{Z}_{t-\tau_{k}}\right)
\]
\[
-\left(f\left(Z_{s-\tau_{1}},..,Z_{s-\tau_{k}}\right)-f\left(\tilde{Z}_{s-\tau_{1}},..,\tilde{Z}_{s-\tau_{k}}\right)\right)
\]
By the same procedure as in the proof of Lemma \ref{lem:Apriori Bounds on delayed SDE},
we can rewrite 
\[
f\left(x_{1},...x_{k}\right)-f\left(y_{1},..,y_{k}\right)=\sum_{i=1}^{d}f\left(x_{1},.,x_{i},y_{i+1},..y_{k}\right)-f\left(x_{1},.,x_{i-1},y_{i},..y_{k}\right)
\]
Since $f$ is continuously differentiable we have that 
\[
f\left(x_{1},...x_{k}\right)-f\left(y_{1},..,y_{k}\right)=\sum_{i=1}^{d}g_{i}(x_{i},y_{i})\cdot\left(x_{i}-y_{i}\right)\,\,\,with\,\,\quad,
\]
\[
g_{i}(x_{i},y_{i})=\int_{0}^{1}\nabla f\left(x_{1},..,tx_{i}+\left(1-t\right)y_{i},y_{i+1},...,y_{k}\right)dy
\]
and we have that $\parallel g\parallel_{\infty}\leq\parallel\nabla f\parallel_{\infty}.$
Moreover, since the derivative $Df$ of $f$ is Hölder continuous
of order $\theta$, we have for all $i\in\left\{ 1,..,k\right\} $
\[
|g_{i}\left(x_{i},y_{i}\right)-g_{i}(\tilde{x}_{i},\tilde{y}_{i})|\leq C_{\theta}\left(|x_{i}-\tilde{x}_{i}|+|y_{i}-\tilde{y}_{i}|\right)^{\theta}.
\]
Therefore, write $\triangle_{t}^{i}=Z_{t-\tau_{i}}-\tilde{Z}_{t-\tau_{i}}$,
and $\triangle_{s,t}^{i}=\triangle_{t}^{i}-\triangle_{s}^{i},$ we
have 
\[
|f\left(Z_{\cdot-\tau_{1}},..,Z_{\cdot-\tau_{k}}\right)_{s,t}-f\left(\tilde{Z}_{\cdot-\tau_{1}},..,\tilde{Z}_{\cdot-\tau_{k}}\right)_{s,t}|
\]
\[
\leq\sum_{i=1}^{k}|g_{i}\left(Z_{t-\tau_{i}},\tilde{Z}_{t-\tau_{i}}\right)\triangle_{t}^{i}-g\left(Z_{s-\tau_{i}},\tilde{Z}_{s-\tau_{i}}\right)\triangle_{s}^{i}|
\]
\[
=\sum_{i=1}^{k}|g_{i}\left(Z_{t-\tau_{i}},\tilde{Z}_{t-\tau_{i}}\right)\triangle_{s,t}^{i}+\left(g_{i}\left(Z_{t-\tau_{i}},\tilde{Z}_{t-\tau_{i}}\right)-g_{i}\left(Z_{s-\tau_{i}},\tilde{Z}_{s-\tau_{i}}\right)\right)\triangle_{s}^{i}|
\]
\[
\leq\sum_{i=1}^{k}\parallel\nabla f\parallel_{\infty}|Z_{s-\tau_{i},t-\tau_{i}}-\tilde{Z}_{s-\tau_{i},t-\tau_{i}}|
\]
\[
+C_{\theta}\left(|Z_{s-\tau_{i},t-\tau_{i}}|+|\tilde{Z}_{s-\tau_{i},t-\tau_{i}}|\right)^{\theta}|Z_{s-\tau_{i}}-\tilde{Z}_{s-\tau_{i}}|,
\]
 we have that 
\[
\left(|Z_{s-\tau_{i},t-\tau_{i}}|+|\tilde{Z}_{s-\tau_{i},t-\tau_{i}}|\right)^{\theta}\leq2^{\theta}|t-s|^{\max\left(\alpha\left(t-\tau_{i}\right),\alpha\left(s-\tau_{i}\right)\right)\times\theta}M^{\theta},
\]
 and we can easily see that for $T\leq1$, and all $i\in\left\{ 1,..,k\right\} $
\[
\sup_{s\in\left[0,T\right]}|Z_{s-\tau_{i}}-\tilde{Z}_{s-\tau_{i}}|\leq|Z_{0}-\tilde{Z}_{0}|+\parallel Z_{\cdot-\tau_{i}}-\tilde{Z}_{\cdot-\tau_{i}}\parallel_{\alpha\left(\cdot-\tau_{i}\right)}.
\]
 Combining the above and using the fact that $\parallel Y_{\cdot-\tau}\parallel_{\alpha\left(\cdot-\tau\right)}\sim\parallel Y_{\cdot}\parallel_{\alpha\left(\cdot\right)}$,
we end up with 
\[
|f\left(Z_{\cdot-\tau_{1}},..,Z_{\cdot-\tau_{k}}\right)_{s,t}-f\left(\tilde{Z}_{\cdot-\tau_{1}},..,\tilde{Z}_{\cdot-\tau_{k}}\right)_{s,t}|
\]
\[
\leq C_{\theta,M,k}\parallel f\parallel_{C_{b}^{1}}\left(|Z_{0}-\tilde{Z}_{0}|+\parallel Z-\tilde{Z}\parallel_{\alpha\left(\cdot\right)}\right)\sum_{i=1}^{k}|t-s|^{\max\left(\alpha\left(t-\tau_{i}\right),\alpha\left(s-\tau_{i}\right)\right)\times\theta}.
\]
\end{proof}
We are now ready to prove Theorem \ref{thm:Delayed Young RDE}. 
\begin{proof}
We follow a classical Picard iteration technique similar to the one
used in \cite{FriHai}, where we will construct the solution in a
space $\mathcal{C}^{\alpha\left(\cdot\right)-\epsilon}\left(\left[0,T\right]\right)$
for some small $\epsilon>0$ such that we still have 
\[
\inf_{t\in\left[\tau_{k},T\right];i\in\left\{ 1,..,k\right\} }\left\{ \alpha\left(t\right)+\alpha\left(t-\tau_{i}\right)-2\epsilon\right\} >1.
\]
The reason for this is that we then get an explicit dependence on
the time time interval $\left[0,T\right]$ , in the sense that for
a path $X\in\mathcal{C^{\alpha\left(\cdot\right)}}\left(\left[0,T\right]\right)$,
we have
\[
\parallel X\parallel_{\alpha\left(\cdot\right)-\epsilon;\left[0,T\right]}\leq T^{\epsilon}\parallel X\parallel_{\alpha\left(\cdot\right);\left[0,T\right]}.
\]
Usually, the fixed point argument is done through constructing a solution
$Y$ on some small time domain $\left[0,T_{0}\right]$ ensuring that
we can apply Banach's fixed point theorem, and then extend this solution
to all of $\left[0,T\right]$ by constructing solutions on $\left[0,T_{0}\right],\left[T_{0},2T_{0}\right],...,\left[nT_{0},T\right]$
starting in $Y_{T_{0}},Y_{2T_{0}},..$ and so on. However, as we are
dealing with delay, the solution to equation (\ref{eq:Delayed RDE})
is trivially given for $t\in\left[0,\tau_{1}\right]$, i.e since $Y_{u}=\xi_{u}$
on $\left[-\tau_{k},0\right],$we get 
\[
Y_{t}=\xi_{0}+\int_{0}^{t}f\left(Y_{r-\tau_{1}},Y_{r-\tau_{2}},..,Y_{r-\tau_{k}}\right)dX_{r}=\xi_{0}+\int_{0}^{t}f\left(\xi_{r-\tau_{1}},..,\xi_{r-\tau_{k}}\right)dX_{r}.
\]
On the next interval when $t\in\left[\tau_{1},\tau_{2}\right]$ we
have 
\[
Y_{t}=Y_{\tau_{1}}^{|\left[0,\tau_{1}\right]}+\int_{\tau_{1}}^{t}f\left(Y_{r-\tau_{1}},Y_{r-\tau_{2}},..,Y_{r-\tau_{k}}\right)dX_{r}
\]
\[
=Y_{\tau_{1}}^{|\left[0,\tau_{1}\right]}+\int_{\tau_{1}}^{t}f\left(Y_{r-\tau_{1}},\xi_{r-\tau_{2}},..,\xi_{r-\tau_{k}}\right)dX_{r}.
\]
 Therefore, we must first consider a fixed point argument on $\left[\tau_{1},\left(\tau_{1}+T_{0}\right)\wedge\tau_{2}\right]$,
where we use the initial value of the solution to be $Y_{\tau_{1}}^{|\left[0,\tau_{1}\right]}$
(i.e the solution restricted to the domain $\left[0,\tau_{1}\right]$),
and then prove existence on the subsequent intervals until we have
existence on all of $\left[\tau_{1},\tau_{2}\right]$. Then we repeat
this procedure for $\left[\tau_{2},\left(\tau_{2}+T_{0}\right)\wedge\tau_{3}\right]$
(for some possibly different $T_{0}$) but here we must show existence
and uniqueness of the equation 
\[
Y_{t}=Y_{\tau_{2}}^{|\left[\tau_{1},\tau_{2}\right]}+\int_{\tau_{2}}^{t}f\left(Y_{r-\tau_{1}},Y_{r-\tau_{2}},\xi_{r-\tau_{3}},..,\xi_{r-\tau_{k}}\right)dX_{r}.
\]
We will do this inductively. We will first prove existence and uniqueness
on all of $\left[\tau_{1},\tau_{2}\right],$ and after that we assume
existence and uniqueness on an interval $\left[\tau_{i-1},\tau_{i}\right]$
and we prove existence on $\left[\tau_{i},\tau_{i+1}\right].$ First
consider the interval $\left[\tau_{1},\left(\tau_{1}+T_{0}\right)\wedge\tau_{2}\right]$
where $T_{0}$ will be chosen such that an appropriate solution map
is a contraction on $\mathcal{C}^{\alpha\left(\cdot\right)-\epsilon}$
for some very small $\epsilon>0$. Again write $f\left(X_{t}\right)-f\left(X_{s}\right)=f\left(X\right)_{s,t}$
for short, and define a mapping 
\[
\mathcal{V}_{T_{0}}^{\tau_{1}}:\mathcal{C}^{\alpha\left(\cdot\right)-\epsilon}\left(\left[\tau_{1},\tau_{1}+T_{0}\wedge\tau_{2}\right]\right)\rightarrow\mathcal{C}^{\alpha\left(\cdot\right)-\epsilon}\left(\left[\tau_{1},\tau_{1}+T_{0}\wedge\tau_{2}\right]\right)
\]
\[
\mathcal{V}_{T_{0}}^{\tau_{1}}\left(Y\right):=\left(Y_{\tau_{1}}^{|\left[0,\tau_{1}\right]}+\int_{\tau_{1}}^{t}f_{\xi}\left(Y_{r-\tau_{1}}\right)dX_{r};t\in\left[\tau_{1},\tau_{1}+T_{0}\wedge\tau_{2}\right]\right)
\]
with 
\[
Y_{\tau_{1}}^{|\left[0,\tau_{1}\right]}=\xi_{0}+\int_{0}^{\tau_{1}}f\left(\xi_{r-\tau_{1}},..,\xi_{r-\tau_{k}}\right)dX_{r},
\]
 for $\xi\in C^{1}\left(\left[-\tau_{k},0\right]\right)$, and to
shorten notation we have used 
\[
f_{\xi}\left(Y_{r-\tau_{1}}\right):=f\left(Y_{r-\tau_{1}},\xi_{r-\tau_{2}},..,\xi_{r-\tau_{k}}\right).
\]
 Observe that for $s,t\in\left[\tau_{1},\tau_{1}+T_{0}\wedge\tau_{2}\right]$
and $Y,\tilde{Y}:\left[\tau_{1},\tau_{1}+T_{0}\wedge\tau_{2}\right]\rightarrow\mathbb{R}^{n},$
\[
|\mathcal{V}_{T_{0}}^{\tau_{1}}\left(Y\right)_{s,t}-\mathcal{V}_{T_{0}}^{\tau_{1}}\left(\tilde{Y}\right)_{s,t}|\leq|\int_{s}^{t}f_{\xi}\left(Y_{r-\tau_{1}}\right)-f_{\xi}\left(\tilde{Y}_{r-\tau_{1}}\right)dX_{r}|
\]
\[
=|\int_{s}^{t}f_{\xi}\left(Y_{r-\tau_{1}}\right)\pm f_{\xi}\left(Y_{s-\tau_{1}}\right)-\left(f_{\xi}\left(\tilde{Y}_{r-\tau_{1}}\right)\pm f_{\xi}\left(\tilde{Y}_{s-\tau_{1}}\right)\right)dX_{r}|
\]
\[
\leq|\int_{s}^{t}f_{\xi}\left(Y\right)_{r-\tau_{1},s-\tau_{1}}-f_{\xi}\left(\tilde{Y}\right)_{r-\tau_{1},s-\tau_{1}}dX_{r}|+|f_{\xi}\left(Y_{s-\tau_{1}}\right)-f_{\xi}\left(\tilde{Y}_{s-\tau_{1}}\right)||X_{s,t}|.
\]
 Combining results from Lemma \ref{prop:Young integrals of variable order}
with 
\[
\beta=\inf_{i\in\left\{ 1,..,k\right\} }\left\{ \alpha\left(t\right)+\alpha\left(t-\tau_{i}\right)-2\epsilon\right\} >1,
\]
 and Lemma \ref{lem:Regularity of composition of functions} with
$\theta=1$, we can see that
\[
|\mathcal{V}_{T_{0}}^{\tau_{1}}\left(Y\right)_{s,t}-\mathcal{V}_{T_{0}}^{\tau_{1}}\left(\tilde{Y}\right)_{s,t}|
\]
\[
\leq\parallel f_{\xi}\left(Y\right)-f_{\xi}\left(\tilde{Y}\right)\parallel_{\alpha\left(\cdot-\tau_{1}\right)-\epsilon;\left[\tau_{1},\tau_{1}+T_{0}\wedge\tau_{2}\right]}\parallel X\parallel_{\alpha\left(\cdot\right)-\epsilon;\left[\tau_{1},\tau_{1}+T_{0}\wedge\tau_{2}\right]}|t-s|^{\beta}
\]
\[
+\parallel f\parallel_{C_{b}^{1}}|Y_{s-\tau}-\tilde{Y}_{s-\tau}||X_{s,t}|
\]
\[
\leq C_{\alpha\left(\cdot\right),M}\parallel f\parallel_{C_{b}^{2}}\left(1+T_{0}^{\beta-1}\right)\left(\parallel Y-\tilde{Y}\parallel_{\alpha\left(\cdot\right)-\epsilon;\left[\tau_{1},\tau_{1}+T_{0}\wedge\tau_{2}\right]}\right)
\]
\[
\times\parallel X\parallel_{\alpha\left(\cdot\right);\left[\tau_{1},\tau_{1}+T_{0}\wedge\tau_{2}\right]}|t-s|^{\max\left(\alpha\left(t\right),\alpha\left(s\right)\right)},
\]
where we ave used the assumption that $Y_{0}=\tilde{Y}_{0}$, and
$\parallel Y_{\cdot-\tau}\parallel_{\alpha\left(\cdot-\tau\right)}\sim\parallel Y_{\cdot}\parallel_{\alpha\left(\cdot\right)}$
. Dividing by $|t-s|^{\max\left(\alpha\left(t\right),\alpha\left(s\right)\right)}$
and taking supremum, we obtain
\[
\parallel\mathcal{V}_{T_{0}}\left(Y\right)-\mathcal{V}_{T_{0}}\left(\tilde{Y}\right)\parallel_{\alpha\left(\cdot\right)-\epsilon;\left[\tau_{1},\tau_{1}+T_{0}\wedge\tau_{2}\right]}
\]
\[
\leq C\left(1+T_{0}^{\beta-1}\right)\parallel X\parallel_{\alpha\left(\cdot\right);\left[\tau_{1},\tau_{1}+T_{0}\wedge\tau_{2}\right]}\parallel Y-\tilde{Y}\parallel_{\alpha\left(\cdot\right)-\epsilon;\left[\tau_{1},\tau_{1}+T_{0}\wedge\tau_{2}\right]}T_{0}^{\epsilon}.
\]
 Choose now $T_{0}>0$ such that 
\[
q:=C\parallel X\parallel_{\alpha\left(\cdot\right);\left[\tau_{1},\tau_{1}+T_{0}\wedge\tau_{2}\right]}\times\left(1+T_{0}^{\beta-1}\right)T_{0}<1,
\]
clearly $\mathcal{V}_{T_{0}}^{\tau_{1}}$ is a contraction on $\mathcal{C}^{\alpha\left(\cdot\right)-\epsilon}\left(\left[\tau_{1},\tau_{1}+T_{0}\wedge\tau_{2}\right]\right)$
for some very small $\epsilon>0$, and we have proved there exists
a $q\in\left(0,1\right)$ such that 
\[
\parallel\mathcal{V}_{T_{0}}^{\tau_{1}}\left(Y\right)-\mathcal{V}_{T_{0}}^{\tau_{1}}\left(\tilde{Y}\right)\parallel_{\alpha\left(\cdot\right)-\epsilon;\left[\tau_{1},\tau_{1}+T_{0}\wedge\tau_{2}\right]}\leq q\parallel Y-\tilde{Y}\parallel_{\alpha\left(\cdot\right)-\epsilon;\left[\tau_{1},\tau_{1}+T_{0}\wedge\tau_{2}\right]},
\]
and the existence of a unique solution on $\left[\tau_{1},\tau_{1}+T_{0}\wedge\tau_{2}\right]$
follows from the Banach fixed point theorem. By standard arguments
for regular differential equations, if $T_{0}<\tau_{2}$ we can iterate
this procedure for intervals $\left[\tau_{1}+kT_{o},\left(\tau_{1}+(k+1)T_{0}\right)\wedge\tau_{2}\right]$
for all $k\leq\sup\left\{ n\in\mathbb{N}|\tau_{1}+nT_{0}<\tau_{2}\right\} $.
As described in the beginning of the proof, assume now that we have
existence and uniqueness on a sub interval $\left[\tau_{i-1},\tau_{i}\right]\subset\left[0,T\right]$,
and we will prove that the equation 
\[
Y_{t}=Y_{\tau_{i}}^{|\left[\tau_{i-1},\tau_{i}\right]}+\int_{\tau_{i}}^{t}f\left(Y_{r-\tau_{1}},..,Y_{r-\tau_{i}},\xi_{r-\tau_{i+1}},...,\xi_{r-\tau_{k}}\right)dX_{r}\,\,\,t\in\left[\tau_{i},\tau_{i+1}\right].
\]
But by the very same procedure as we did for the interval $\left[\tau_{1},\tau_{1}+T_{0}\wedge\tau_{2}\right]$,
we can construct a mapping $\mathcal{V}_{T_{0}}^{\tau_{i}}:\mathcal{C}^{\alpha\left(\cdot\right)-\epsilon}\left(\left[\tau_{i},\tau_{i}+T_{0}\wedge\tau_{i+1}\right]\right)\rightarrow\mathcal{C}^{\alpha\left(\cdot\right)-\epsilon}\left(\left[\tau_{i},\tau_{i}+T_{0}\wedge\tau_{i+1}\right]\right),$
and notice that $T_{0}$ may be different now than what it was on
the interval $\left[\tau_{1},\tau_{1}+T_{0}\wedge\tau_{2}\right]$.
By the same calculations as above, we get that for two paths $Y,\tilde{Y}:\left[\tau_{i},\tau_{i}+T_{0}\wedge\tau_{i+1}\right]\rightarrow\mathbb{R}^{n}$
we have 
\[
\parallel\mathcal{V}_{T_{0}}^{\tau_{i}}\left(Y\right)-\mathcal{V}_{T_{0}}^{\tau_{i}}\left(\tilde{Y}\right)\parallel_{\alpha\left(\cdot\right)-\epsilon;\left[\tau_{i},\tau_{i}+T_{0}\wedge\tau_{i+1}\right]}\leq q\parallel Y-\tilde{Y}\parallel_{\alpha\left(\cdot\right)-\epsilon;\left[\tau_{i},\tau_{i}+T_{0}\wedge\tau_{i+1}\right]}
\]
 for some $q\in\left(0,1\right)$, and we again conclude that $\mathcal{V}_{T_{0}}^{\tau_{i}}$
is a contraction mapping on $\mathcal{C}^{\alpha\left(\cdot\right)-\epsilon}\left(\left[\tau_{i},\tau_{i}+T_{0}\wedge\tau_{i+1}\right]\right)$
and existence and uniqueness follows from Banach's fixed point theorem.
By standard procedure of iteration of the solution, we can extend
the solution to the whole line $\left[\tau_{i},\tau_{i+1}\right]$.
With this induction argument we can conclude that the the solution
exists and is unique on all intervals made up by the delay times $\left\{ \tau_{i}\right\} _{i=1}^{k}$
including the interval $\left[\tau_{k},T\right]$. The extension to
a unique solution on all of $\mathcal{C}^{\alpha\left(\cdot\right)-\epsilon}\left(\left[0,T\right]\right)$
then follows from Lemma \ref{lem: Scaling of H=0000F6lder norms}.
At last, from the inequlaity 
\[
|\int_{s}^{t}f\left(Y_{t-\tau_{1}},Y_{t-\tau_{2}},..,Y_{t-\tau_{k}}\right)dX_{t}|\leq\parallel f\parallel_{C_{b}^{2}}|X_{s,t}|+\mathcal{O}\left(|t-s|^{\beta}\right),
\]
it follows that the solution acctually lives in $\mathcal{C}^{\alpha\left(\cdot\right)}$
as the regularity is inherited by the driver $X\in\mathcal{C}^{\alpha\left(\cdot\right)}.$
\end{proof}

\section{\label{sec:Variable-order-rough path and mBm}Variable order rough
paths -examples from multifractional processes. }

In this section we will discuss the construction of a variable order
rough path and as an example construct the iterated integral for a
multifractional Brownian motion. The discussion evolves around the
problem that the regularity function of a stochastic processes is
desirably evolving both above and below $\frac{1}{2}$, and thus the
construction of an iterated integral (and in any case also the need..)
is different depending on the level of the regularity function. We
will show an algorithm to construct the iterated integral, and apply
this to lift the multifractional Brownian motion to a rough path. 

\subsection{Definition and construction of iterated integrals}

We first present the definition of a variable order rough path, although
it is similar to a regular rough path known from f.ex (\cite{FriHai,FriVic,Gubinelli}). 
\begin{defn}
\label{def:variable order rough pah}$\left(Rough\,\,\,path\,\,\,of\,\,\,variable\,\,\,order\right)$
Let $\alpha:\left[0,T\right]\rightarrow\left(\frac{1}{3},1\right)$
and assume that for a path $X\in\mathcal{C^{\alpha\left(\cdot\right)}}\left(\left[0,T\right];\mathbb{R}^{d}\right)$
there exists an object $\mathbb{X}\in\mathcal{C}_{2}^{2\alpha\left(\cdot\right)\wedge1}\left(\left[0,T\right]^{2};\mathbb{R}^{d\times d}\right)$
such that 
\[
\delta\left(\mathbb{X}\right)_{s,u,t}=X_{s,u}X_{s,t}.
\]
We then say that $\mathbf{X}=\left(X,\mathbb{X}\right)$ is a rough
path of variable order, and denote by $\mathscr{C}^{\alpha\left(\cdot\right)}\left(\left[0,T\right];\mathbb{R}^{d}\right)$
the space of all such rough paths. 
\end{defn}

\begin{rem}
We usually call $\mathbb{X}$ the iterated integral of $X$ as it
behaves algebraically and analytically as such. However, exactly how
$\mathbb{X}$ is constructed is not important in general for the rough
path theory, as long as we know that it exists. Also notice the regularity
requirement on $\mathbb{X}$ is given in the variable order Hölder
norm of order $2\alpha\left(\cdot\right)\wedge1$. This is due to
the fact that although locally $2\alpha$ may be above $1$, the regularity
above $1$ will not increase the regularity of the rough differential
equations, or other applications we consider in this article. For
simplicity we use this assumption, as we will find it to be well suited
in construction of the variable order rough path. 
\end{rem}

When discussing rough paths of variable order, it may not be clear
how to construct an iterated integral suited for the theory of rough
paths. It is often easier to construct iterated integrals on very
small squares, using the local behavior on that small square, than
to construct the iterated integral on the whole space. The next theorem
shows that the construction of an iterated integral $\mathbb{X}$
on $\left[0,T\right]^{2}$ is completely determined from how $\mathbb{X}$
behaves locally around the diagonal of $\left[0,T\right]^{2}.$ That
is, if we consider a partition, $\mathcal{P}=\left\{ \left[\rho_{0},\rho_{1}\right],..,\left[\rho_{N-1},\rho_{N}\right]\right\} $,
and know that on each square $\left[\rho_{i},\rho_{i+1}\right]^{2}$
there exists an object $\mathbb{X}$ which behaves like an iterated
integral (analytically and algebraically) on each square, then these
iterated integrals can be ``glued'' together to construct an iterated
integral on the whole square $\left[0,T\right]^{2}$. 
\begin{thm}
Consider the partition $\mathcal{P}=\left\{ \left[\rho_{0},\rho_{1}\right],..,\left[\rho_{N-1},\rho_{N}\right]\right\} $
of $\left[0,T\right]$, and assume that on each square $\left[\rho_{i},\rho_{i+1}\right]^{2}$
there exists a corresponding object $\mathbb{X}^{|\left[\rho_{i},\rho_{i+1}\right]}$
satisfying locally the Chen relation and the analytic requirement.
Then there exists an object $\mathbb{X}:\left[0,T\right]^{2}\rightarrow\mathbb{R}^{d\times d}$
such that 
\[
\delta\left(\mathbb{X}\right)_{s,u,t}=X_{s,u}X_{s,t},\,\,\,and\,\,\,\parallel\mathbb{X}\parallel_{2\alpha\left(\cdot\right)\wedge1}<\infty,
\]
 and it follows that $\left(X,\mathbb{X}\right)\in\mathscr{C}_{2}^{\alpha\left(\cdot\right)}\left(\left[0,T\right];\mathbb{R}^{d}\right)$. 
\end{thm}

\begin{proof}
Consider the rectangle $\left[\rho_{i},\rho_{i+1}\right]\times\left[\rho_{j},\rho_{j+1}\right]$
for $0\leq i<j\leq N$. For $s\in\left[\rho_{i},\rho_{i+1}\right]$
and $t\in\left[\rho_{j},\rho_{j+1}\right]$ define $\mathcal{P}\left[s,t\right]=\mathcal{P}\cap\left[s,t\right]$,
i.e. $\mathcal{P}\cap\left[s,t\right]=\left\{ \left[s,\rho_{i+1}\right],\left[\rho_{i+1},\rho_{i+2}\right],..,\left[\rho_{j},t\right]\right\} $,
and define 
\[
\mathbb{X}_{s,t}:=\sum_{\left[u,v\right]\in\mathcal{P}\left[s,t\right]}\mathbb{X}_{u,v}^{|\left[u,v\right]}+X_{s,u}X_{u,v}.
\]
Consider now $s<r<t$ with $s\in\left[\rho_{i},\rho_{i+1}\right]$,
$r\in$$\left[\rho_{k},\rho_{k+1}\right]$ and $t\in\left[\rho_{j},\rho_{j+1}\right]$
for $i\leq k\leq j$, then 
\[
\left(\delta\mathbb{X}\right)_{s,r,t}=\sum_{\left[u,v\right]\in\mathcal{P}\left[s,t\right]}\mathbb{X}_{u,v}^{|\left[u,v\right]}+X_{s,u}X_{u,v}-\left(\sum_{\left[u,v\right]\in\mathcal{P}\left[s,r\right]}\mathbb{X}_{u,v}^{|\left[u,v\right]}+X_{s,u}X_{u,v}\right)
\]
 
\[
-\left(\sum_{\left[u,v\right]\in\mathcal{P}\left[r,t\right]}\mathbb{X}_{u,v}^{|\left[u,v\right]}+X_{r,u}X_{u,v}\right)
\]
 Let us divide the argument into treating first the terms involving
$\mathbb{X}$ and secondly, we look at the product of increments of
$X$, and at last we combine them. First, observe that 
\[
\sum_{\left[u,v\right]\in\mathcal{P}\left[s,t\right]}\mathbb{X}_{u,v}^{|\left[u,v\right]}=\sum_{\left[u,v\right]\in\mathcal{P}\left[s,\rho_{k}\right]}\mathbb{X}_{u,v}^{|\left[u,v\right]}+\mathcal{\mathbb{X}}_{\rho_{k},\rho_{k+1}}+\sum_{\left[u,v\right]\in\mathcal{P}\left[\rho_{k+1},t\right]}\mathbb{X}_{u,v}^{|\left[u,v\right]}.
\]
Therefore, looking at the difference $\mathbb{X}_{s,t}-\mathbb{X}_{s,r}$,
we have 
\[
\sum_{\left[u,v\right]\in\mathcal{P}\left[s,t\right]}\mathbb{X}_{u,v}^{|\left[u,v\right]}-\sum_{\left[u,v\right]\in\mathcal{P}\left[s,r\right]}\mathbb{X}_{u,v}^{|\left[u,v\right]}
\]
\[
=\mathcal{\mathbb{X}}_{\rho_{k},\rho_{k+1}}-\mathbb{X}_{\rho_{k},r}^{|\left[\rho_{k},\rho_{k+1}\right]}+\sum_{\left[u,v\right]\in\mathcal{P}\left[\rho_{k+1},t\right]}\mathbb{X}_{u,v}^{|\left[u,v\right]}
\]
\[
=\mathbb{X}_{r,\rho_{k+1}}^{|\left[\rho_{k},\rho_{k+1}\right]}+X_{\rho_{k},r}X_{r,\rho_{k+1}}+\sum_{\left[u,v\right]\in\mathcal{P}\left[\rho_{k+1},t\right]}\mathbb{X}_{u,v}^{|\left[u,v\right]}.
\]
And in the same way we find that when subtracting $\mathbb{X}_{r,t}$,
we get 
\[
\sum_{\left[u,v\right]\in\mathcal{P}\left[s,t\right]}\mathbb{X}_{u,v}^{|\left[u,v\right]}-\sum_{\left[u,v\right]\in\mathcal{P}\left[s,r\right]}\mathbb{X}_{u,v}^{|\left[u,v\right]}-\sum_{\left[u,v\right]\in\mathcal{P}\left[r,t\right]}\mathbb{X}_{u,v}^{|\left[u,v\right]}
\]
 
\[
=X_{\rho_{k},r}X_{r,\rho_{k+1}}.
\]
Furthermore, for the product of increments, we have that 
\[
\sum_{\left[u,v\right]\in\mathcal{P}\left[s,t\right]}X_{s,u}X_{u,v}-\sum_{\left[u,v\right]\in\mathcal{P}\left[s,r\right]}X_{s,u}X_{u,v}-\sum_{\left[u,v\right]\in\mathcal{P}\left[r,t\right]}X_{r,u}X_{u,v}
\]
 
\[
=X_{s,\rho_{k}}X_{r,\rho_{k+1}}+X_{s,r}\times\sum_{\left[u,v\right]\in\mathcal{P}\left[\rho_{k+1},t\right]}X_{u,v}.
\]
combining our findings, we can see that 

\[
\delta\left(\mathbb{X}\right)_{s,r,t}=X_{\rho_{k},r}X_{r,\rho_{k+1}}+X_{s,\rho_{k}}X_{r,\rho_{k+1}}
\]
\[
+X_{s,r}\times\sum_{\left[u,v\right]\in\mathcal{P}\left[\rho_{k+1},t\right]}X_{u,v}=X_{s,r}\times\sum_{\left[u,v\right]\in\mathcal{P}\left[r,t\right]}X_{u,v}=X_{s,r}X_{r,t},
\]
and we can conclude that this construction of the iterated integral
satisfies the Chen's relation. Next we must prove that this construction
is of sufficient variable Hölder regularity. 

We can see that 
\[
\parallel\mathbb{X}_{s,t}\parallel_{2\alpha\left(\cdot\right)\wedge1;\left[0,T\right]}=\sup_{0\leq s<t\leq T}\frac{|\sum_{\left[u,v\right]\in\mathcal{P}\left[s,t\right]}\mathbb{X}_{u,v}^{|\left[u,v\right]}+X_{s,u}X_{u,v}|}{|t-s|^{2\alpha\left(s\right)\wedge1}}
\]
\[
\leq\sup_{0\leq s<t\leq T}\sum_{\left[u,v\right]\in\mathcal{P}\left[s,t\right]}\frac{|\mathbb{X}_{u,v}^{|\left[u,v\right]}|}{|t-s|^{2\alpha\left(s\right)\wedge1}}+\sup_{0\leq s<t\leq T}\sum_{\left[u,v\right]\in\mathcal{P}\left[s,t\right]}\frac{|X_{s,u}X_{u,v}|}{|t-s|^{2\alpha\left(s\right)\wedge1}}.
\]
Treating each term separately we can see for the first term, 
\[
\sup_{0\leq s<t\leq T}\sum_{\left[u,v\right]\in\mathcal{P}\left[s,t\right]}\frac{|\mathbb{X}_{u,v}^{|\left[u,v\right]}|}{|t-s|^{2\alpha\left(s\right)\wedge1}}
\]
\[
\leq\sup_{0\leq s<t\leq T}\sum_{\left[u,v\right]\in\mathcal{P}\left[s,t\right]}\frac{\parallel\mathbb{X}^{|\left[u,v\right]}\parallel_{2\alpha\left(\cdot\right)\wedge1;\left[u,v\right]}|v-u|^{2\alpha\left(u\right)\wedge1}}{|t-s|^{2\alpha\left(s\right)\wedge1}}
\]
\[
\leq C\sup_{0\leq s<t\leq T}\sum_{\left[u,v\right]\in\mathcal{P}\left[s,t\right]}\frac{|v-u|^{2\alpha\left(u\right)\wedge1}}{|t-s|^{2\alpha\left(s\right)\wedge1}},
\]
Divide the supreme such that 
\[
\sup_{0\leq s<t\leq T}=\sup_{0\leq s<t\leq T;|t-s|\leq|\mathcal{P}|_{*}}+\sup_{0\leq s<t\leq T;|t-s|\geq|\mathcal{P}|_{*}}
\]
where $|\mathcal{P}|_{*}=\inf\left\{ |v-u||\left[u,v\right]\in\mathcal{P}\right\} $
is a fixed number, and observe that 
\[
\sup_{0\leq s<t\leq T;|t-s|\leq|\mathcal{P}|_{*}}\sum_{\left[u,v\right]\in\mathcal{P}\left[s,t\right]}\frac{|v-u|^{2\alpha\left(u\right)\wedge1}}{|t-s|^{2\alpha\left(s\right)\wedge1}}=1.
\]
 Furthermore, it is clear that there exists a constant $C_{\alpha,T}$
such that 
\[
\sup_{0\leq s<t\leq T;|t-s|\geq|\mathcal{P}|_{*}}\sum_{\left[u,v\right]\in\mathcal{P}\left[s,t\right]}\frac{|v-u|^{2\alpha\left(u\right)\wedge1}}{|t-s|^{2\alpha\left(s\right)\wedge1}}\leq C_{\alpha,T}<\infty,
\]
and hence 
\[
\sup_{0\leq s<t\leq T}\sum_{\left[u,v\right]\in\mathcal{P}\left[s,t\right]}\frac{|\mathbb{X}_{u,v}^{|\left[u,v\right]}|}{|t-s|^{2\alpha\left(s\right)\wedge1}}\leq C_{\alpha,T}.
\]
For the second part, we have 
\[
\sum_{\left[u,v\right]\in\mathcal{P}\left[s,t\right]}\frac{|X_{s,u}X_{u,v}|}{|t-s|^{2\alpha\left(s\right)\wedge1}}\leq\parallel X\parallel_{\alpha\left(\cdot\right);\left[0,T\right]}^{2}\sum_{\left[u,v\right]\in\mathcal{P}\left[s,t\right]}\frac{|u-s|^{2\alpha\left(s\right)}|v-u|^{2\alpha\left(u\right)}}{|t-s|^{2\alpha\left(s\right)\wedge1}},
\]
Notice that if $|t-s|\leq|\mathcal{P}|_{*},$then the above sum is
zero . On the other hand, there exists a constant $C_{\alpha,T}$
such that 
\[
\sup_{0\leq s<t\leq T;|t-s|\geq|\mathcal{P}|_{*}}\sum_{\left[u,v\right]\in\mathcal{P}\left[s,t\right]}\frac{|u-s|^{2\alpha\left(s\right)}|v-u|^{2\alpha\left(u\right)}}{|t-s|^{2\alpha\left(s\right)\wedge1}}\leq C_{\alpha,T}.
\]
 Combining these results, we can see that 
\[
\parallel\mathbb{X}_{s,t}\parallel_{2\alpha\left(\cdot\right)\wedge1;\left[0,T\right]}\leq C_{\alpha,T},
\]
which concludes our proof. 
\end{proof}

\subsection{Rough path lift of multifractional Brownian motion }
\begin{defn}
Consider a two-sided $d-$dimensional Brownian motion $\left\{ B_{t}\right\} _{t\in(-\infty,T]}$
and let $h:\left[0,T\right]\rightarrow\left(0,1\right)$ be a regularity
function satisfying $\sup_{s,r\in\left[0,T\right]}\left(r-s\right)^{h(r)-h(s)}\leq C$.
We define the multifractional Brownian motion $(mBm)$ by 
\[
B_{t}^{h}=\int_{-\infty}^{0}\left(t-r\right)^{h\left(t\right)-\frac{1}{2}}-\left(-r\right)^{h(t)-\frac{1}{2}}dB_{r}+\int_{0}^{t}\left(t-r\right)^{h(t)-\frac{1}{2}}dB_{r}.
\]
 
\end{defn}

The above definition is given as a generalization of the fractional
Brownian motion in the Mandelbrot-Van-Ness representation. This representation
is equal in distribution (up to a deterministic function) to the Harmonizable
representation of the mBm given by 
\[
\tilde{B}_{t}^{h}=\int_{\mathbb{R}}\frac{e^{it\xi}-1}{i\xi|\xi|^{h(t)-\frac{1}{2}}}dW\left(\xi\right),
\]
where $W\left(\xi\right)$ is a complex valued Brownian motion. We
will in the rest of this section assume the representation of the
mBm is given as in the Mandelbrot-Van-Ness case. The next proposition
gives results relating to the co variance and local regularities of
the multifractional Brownian motion. 
\begin{prop}
\label{prop:(Covariance-of-RLmBm)}(Co-variance of mBm) Let $\left\{ \tilde{B}_{t}^{h}\right\} $
be a mBm be defined as above. Then for $0\leq s\leq t\leq T$, we
have 
\[
\tilde{R}^{h}\left(t,s\right)=E\left[B_{t}^{h}\cdot B_{s}^{h}\right]=c(h(t),h(s))\left(t^{2h(t)}+s^{2h(s)}-|t-s|^{h(t)+h(s)}\right)
\]
 where 
\[
c(x,y)=\frac{\sqrt{\Gamma\left(2x+1\right)\Gamma\left(2y+1\right)\sin\left(\pi x\right)\sin\left(\pi y\right)}}{2\Gamma\left(x+y+1\right)\sin\left(\pi\left(x+y\right)/2\right)}.
\]
 
\end{prop}

When discussing fractional and multifractional Brownian motions, their
Hölder regularities is of great interest. 
\begin{prop}
Let $\left\{ B_{t}^{h}\right\} _{t\in[0,T]}$ be a mBm, with regularity
function $h\in C^{1}\left(\left[0,T\right]\right)$ . Then we have
global regularity 
\[
|B_{t}^{h}-B_{s}^{h}|\leq C|t-s|^{\min(\frac{1}{2},h_{*})}\,\,\,for\,\,\,all\,\,(t,s)\in\triangle^{(2)}\left(\left[0,T\right]\right)
\]
 and local regularity 
\[
\lim_{u\rightarrow0}\frac{|B_{t+u}^{h}-B_{t}^{h}|}{|u|^{h(t)}}<\infty\,\,\,P-a.s.\,\,\forall t\in[0,T],
\]
and it follows that $B^{h}\in\mathcal{C}^{h\left(\cdot\right)}\left(\left[0,T\right];\mathbb{R}^{d}\right)$
$P-$a.s. 
\end{prop}

\begin{proof}
Both the claims are thoroughly proved in the seminal paper \cite{Peltier}
by Peltier and Vehèl. 
\end{proof}
For a longer discussion on the properties of the multifractional Brownian
motion, we refer to \cite{AyCoVe,BenJafRed,BouDozMar,Corlay,LebovitsLevy,LebVehHer,SCLim,SergCo}.\\
\\

We will now propose a general algorithm to construct iterated integrals
with respect to a variable order path. We will use a method to divide
the construction of the iterated integral into the case of $\alpha>\frac{1}{2}$,
and $\alpha\leq\frac{1}{2}$, and use familiar methods of construction
in each case. As we are assuming $\alpha$ to be $C^{1}$ , we have
that $\alpha$ fluctuates above and below $\frac{1}{2}$ only a finite
number of times on a finite time interval. Therefore, we want to construct
iterated integrals of the multifractional Brownian motion separately
on these intervals, and then glue them together afterwards. Considering
the fact that the iterated integral of a process $X:\left[0,T\right]\rightarrow\mathbb{R}^{d}$
has $\left[0,T\right]^{2}$ as domain, we must be a bit careful when
we construct iterated integrals. In general, we will use a decomposition
of the iterated integral in the sense that 
\[
\mathbb{X}_{s,t}=\delta\left(\int_{0}^{\cdot}X_{r}dX_{r}\right)_{s,t}-X_{s}X_{s,t}
\]
\[
=\sum_{\left[u,v\right]\in\mathcal{P}}\delta\left(\int_{0}^{\cdot}X_{r}dX_{r}\right)_{u,v}-X_{u}X_{u,v}+X_{s,u}X_{u,v}
\]
where $\delta g_{s,t}:=g_{t}-g_{s}$, and we write $\mathbb{X}^{|\left[u,v\right]}$
for the above iterated integral such that the process $X$ is restricted
to $\left[u,v\right]\subset\left[0,T\right]$. Without loss of generality,
assume $\alpha$ crosses $\frac{1}{2}$ $N$ times and let $\left\{ \rho_{i}\right\} _{i=1}^{N}$
denote a partition of $\left[0,T\right]$ with $\rho_{0}=0$ and $\rho_{N}=T$
and the sequence of subsets $\left[\rho_{0},\rho_{1}\right],...,(\rho_{N-1},\rho_{N}]$
is such that on every even numbered subset in the sequence, the process
$X^{|(\rho_{2k},\rho_{2k+1}]}$ has $\alpha_{t}>\frac{1}{2}$ for
$t\in(\rho_{2k},\rho_{2k+1}]$, and on every odd numbered subset $(\rho_{2k-1},\rho_{2k}]$,
the process $X^{|(\rho_{2k-1},\rho_{2k}]}$ have $\alpha_{t}\leq\frac{1}{2}$
for $t\in(\rho_{2k-1},\rho_{2k}]$ for $k=1,..,\frac{N}{2}$. We then
need to construct the iterated integral on the intervals where $h\leq\frac{1}{2}.$
In this article, we will use the method proposed for fractional Brownian
motion in \cite{NuaTInd}. 
\begin{lem}
\label{lem:MbM rough path}Let $\left[a,b\right]\subset\left[0,T\right]$
and let $\left\{ B_{t}\right\} _{t\in\left[a,b\right]}$ be the Brownian
motion driving $\left\{ B_{t}^{h}\right\} _{t\in\left[a,b\right]}$
the mBm on a probability space $\left(\Omega,P,\mathcal{F}\right)$
with regularity function $h\in C^{1}\left(\left[a,b\right];[h_{*},h^{*}]\right)$
. Define 
\[
\mathbb{B}_{s,t}:=\delta\left(\int_{0}^{\cdot}\left(\cdot-r\right)^{h\left(t\right)-\frac{1}{2}}B_{r}^{h}(i)dB_{r}(j)\right)_{s,t}-B_{s}^{h}(i)B_{s,t}^{h}(j).
\]
Then $\mathbb{B}_{s,t}\in L^{2}\left(\Omega;\mathbb{R}^{d\times d}\right)$
for all $s,t\in\left[0,T\right]^{2}$, and 
\[
E\left[|\mathbb{B}_{s,t}|^{2}\right]\lesssim|t-s|^{4\max\left(h(t),h(s)\right)}
\]
 
\end{lem}

\begin{proof}
First, write $K\left(t,r\right)=\left(t-r\right)^{h(t)-\frac{1}{2}}1_{\left[r\leq t\right]}$
and $\delta K_{s,t}(r)=K(t,r)-K(s,r)$ ,and see that for component
$i,j\in\left\{ 1,...,d\right\} $
\[
\mathbb{B}_{s,t}^{i,j}=\delta\left(\int_{0}^{\cdot}K(\cdot,r)B_{r}^{h}(i)dB_{r}(j)\right)_{s,t}-B_{s}^{h}(i)B_{s,t}^{h}(j)
\]
In the case when $i=j$ the Stratonovich and Itô integral coincide,
and we therefore treat this case first. By the Itô isometry we have 

\[
E\left[\left(\mathbb{B}_{s,t}^{i,j}\right)^{2}\right]
\]

\[
=E\left[\left(\int_{s}^{t}K(t,r)B_{s,r}^{h}(i)dB_{r}(j)\right)^{2}\right]+E\left[\left(\int_{0}^{s}\delta K_{s,t}\left(r\right)B_{r,s}^{h}(i)dB_{r}(j)\right)^{2}\right]
\]
\[
=E\left[\left(\int_{s}^{t}K(t,r)\int_{s}^{r}K(r,u)dB_{u}(i)dB_{r}(j)\right)^{2}\right]
\]
\[
+E\left[\left(\int_{s}^{t}K(t,r)\int_{0}^{s}\delta K_{s,r}\left(u\right)dB_{u}(i)dB_{r}\left(j\right)\right)^{2}\right]
\]
\[
+E\left[\left(\int_{0}^{s}\delta K_{s,t}\left(r\right)\int_{r}^{s}K(s,u)dB_{u}(i)dB_{r}(j)\right)^{2}\right]
\]
\[
+E\left[\left(\int_{0}^{s}\delta K_{s,t}\left(r\right)\int_{0}^{r}\delta K_{r,s}\left(u\right)dB_{u}(i)dB_{r}(j)\right)^{2}\right]
\]
 
\[
=:\sum_{k=1}^{4}I_{k}
\]
 Treating each term separately , we get for $I_{1}$ 
\[
I_{1}=\int_{s}^{t}\int_{s}^{r}\left(t-r\right)^{2h(t)-1}\left(r-u\right)^{2h(r)-1}dudr
\]
\[
=\int_{s}^{t}\left(t-r\right)^{2h(t)-1}\left(r-s\right)^{2h(r)}dr.
\]
set $r=s+\tau\left(t-s\right)$, and see that 
\[
=\int_{0}^{1}\left(t-s-\tau\left(t-s\right)\right)^{2h(t)-1}\tau^{2h(s+\tau\left(t-s\right))}\left(t-s\right)^{2h(s+\tau\left(t-s\right))}\left(t-s\right)d\tau
\]
\[
=\left(t-s\right)^{2h(t)}\int_{0}^{1}\left(t-s\right)^{2h(s+\tau\left(t-s\right))}\left(1-\tau\right)^{2h\left(t\right)-1}\tau^{2h(s+\tau\left(t-s\right))}d\tau
\]
\[
\leq C\left(t-s\right)^{4\max\left(h(t),h(s)\right)}\tilde{B}\left(2h_{*},1-2h^{*}\right),
\]
where $\tilde{B}$ is the Beta function and we have used the assumption
on $h$, 
\[
\sup_{s,r\in\left[0,T\right]}\left(r-s\right)^{h(r)-h(s)}\leq C.
\]
 Next check 
\[
I_{2}=\int_{s}^{t}\left(t-r\right)^{2h(t)-1}\int_{0}^{s}\left(\left(r-u\right)^{h(r)-\frac{1}{2}}-\left(s-u\right)^{h(s)-\frac{1}{2}}\right)^{2}dudr,
\]
and consider the ``inner integral'' first and set $s-u=l$ and $l=\left(r-s\right)y$
then 
\[
\int_{0}^{s}\left(\left(r-u\right)^{h(r)-\frac{1}{2}}-\left(s-u\right)^{h(s)-\frac{1}{2}}\right)^{2}du
\]
\[
=\int_{0}^{s/\left(r-s\right)}\left(\left(r-s\right)^{h(r)-\frac{1}{2}}\left(1+y\right)^{h(r)-\frac{1}{2}}-\left(r-s\right)^{h(s)-\frac{1}{2}}y^{h(s)-\frac{1}{2}}\right)^{2}\left(r-s\right)dy
\]
 
\[
\leq C_{T}\left(r-s\right)^{2h(s)}\int_{0}^{\infty}\left(\left(r-s\right)^{h(r)-h(s)}\left(1+y\right)^{h(r)-\frac{1}{2}}-y^{h(s)-\frac{1}{2}}\right)^{2}dy
\]
Notice that by assumption on $h$, $\sup_{s,r\in\left[0,T\right]}\left(r-s\right)^{h(r)-h(s)}\leq C,$
and the fact that $h(t)\in\left[h_{*},h^{*}\right]\subset\left(0,1\right)$
and hence there exists a constant $C_{T}$ s.t. 
\[
\int_{0}^{\infty}\left(\left(r-s\right)^{h(r)-h(s)}\left(1+y\right)^{h(r)-\frac{1}{2}}-y^{h(s)-\frac{1}{2}}\right)^{2}dy\leq C_{T}.
\]
Combining this, we get 
\[
I_{2}\leq C_{T}\int_{s}^{t}\left(t-r\right)^{2h(t)-1}\left(r-s\right)^{2h(s)}dr,
\]
using similar estimates as for $I_{1}$, we get 
\[
I_{2}\lesssim\left(t-s\right)^{4\max\left(h(t),h(s)\right)}.
\]
 Next consider $I_{3},$and see that 
\[
I_{3}=\int_{0}^{s}\left(\left(t-r\right)^{h(t)-\frac{1}{2}}-\left(s-r\right)^{h(s)-\frac{1}{2}}\right)^{2}\int_{r}^{s}\left(s-u\right)^{2h(s)-1}dudr
\]
\[
=\int_{0}^{s}\left(\left(t-r\right)^{h(t)-\frac{1}{2}}-\left(s-r\right)^{h(s)-\frac{1}{2}}\right)^{2}\left(s-r\right)^{2h(s)}dr.
\]
Again, set $s-r=l$ and $l=\left(t-s\right)x$ we get 
\[
\leq\left(t-s\right)^{4\max\left(h(t),h(s)\right)}\int_{0}^{\infty}\left(\left(1+x\right)^{h(t)-\frac{1}{2}}-\left(t-s\right)^{h(s)-h(t)}x^{h(s)-\frac{1}{2}}\right)^{2}dx,
\]
where the convergence of the integral was proven for $I_{2}.$ At
last, we look at 
\[
I_{4}=\int_{0}^{s}\left(\left(t-u\right)^{h(t)-\frac{1}{2}}-\left(s-u\right)^{h(s)-\frac{1}{2}}\right)^{2}
\]
\[
\times\int_{0}^{r}\left(\left(s-u\right)^{h(s)-\frac{1}{2}}-\left(r-u\right)^{h(r)-\frac{1}{2}}\right)^{2}dudr,
\]
 and use the the results obtained from $I_{2}$ and $I_{3}$ to see
that 
\[
I_{4}\lesssim|t-s|^{4\max\left(h(t),h(s)\right)}.
\]
 Next we must check the case when $i=j$, i.e 
\[
\mathbb{B}_{s,t}^{i,i}=\delta\left(\int_{0}^{\cdot}K(\cdot,r)B_{r}^{h}(i)dB_{r}(i)\right)_{s,t}-B_{s}^{h}(i)B_{s,t}^{h}(i)
\]
 where the integral is of Itô type. Using the convention $K\left(t,r\right)=K(t,r)1_{\left[r\leq t\right]}$,
we have the Itô $\left(\partial\right)$-Stratonovich$\left(d\right)$conversion
( see Lemma 2.4 in \cite{NuaTInd}) we get 
\[
\delta\left(\int_{0}^{\cdot}K(\cdot,r)\int_{0}^{r}K\left(r,u\right)dB_{u}(i)dB_{r}(i)\right)_{s,t}
\]
\[
=\delta\left(\int_{0}^{\cdot}K(\cdot,r)\int_{0}^{r}K\left(r,u\right)\partial B_{u}(i)\partial B_{r}(i)\right)_{s,t}
\]
\[
+\delta\left(\int_{0}^{\cdot}K(\cdot,r)\int_{0}^{r}K\left(r,u\right)dudr\right)_{s,t},
\]
and we can therefore write 
\[
\mathbb{B}_{s,t}^{i,i}=M_{s,t}+V_{s,t}-B_{s}^{h}(i)B_{s,t}^{h}(i).
\]
 The term $M_{s,t}-B_{s}^{h}(i)B_{s,t}^{h}(i)$ can be shown to satisfy
$E\left[\left(M_{s,t}-B_{s}^{h}(i)B_{s,t}^{h}(i)\right)^{2}\right]\lesssim|t-s|^{4\max\left(h(t),h(s)\right)}$
just as we did in the first art with Stratonovich integrals. It therefore
remains to show that $|\delta\left(\int_{0}^{\cdot}K(\cdot,r)\int_{0}^{r}K\left(r,u\right)dudr\right)_{s,t}|\lesssim|t-s|^{4\max\left(h(t),h(s)\right)},$
but this is also clear from the calculations of $I_{4}$. 
\end{proof}
Now that we have a construction of the iterated integral of the multifractional
Brownian motion on all intervals where its regularity function takes
values below or equal to $\frac{1}{2}$, we will now consider when
the regularity function takes values above $\frac{1}{2}$. We give
the following proposition, but we drop the notation of $h$ in $B^{h}$
to simplify notation.
\begin{prop}
Consider an interval $(\rho_{2k},\rho_{2k+1}]$ as defined above such
that the process $B^{|(\rho_{2k},\rho_{2k+1}]}$ has regularity function
$\alpha_{t}>\frac{1}{2}$ for $t\in(\rho_{2k},\rho_{2k+1}]$. Then
we define
\[
\mathbb{B}_{u,v}^{(\rho_{2k},\rho_{2k+1}]}:=\delta\left(\int_{0}^{\cdot}B_{r}dB_{r}\right)_{u,v}-B_{u}B_{u,v}\,\,\,for\,\,\,u,v\in(\rho_{2k},\rho_{2k+1}]
\]
 where $\delta\left(\int_{0}^{\cdot}B_{r}dB_{r}\right)_{u,v}$ is
interpreted in a Young sense, and we have 
\[
\parallel\mathbb{B}_{u,v}^{(\rho_{2k},\rho_{2k+1}]}\parallel_{2\alpha\left(\cdot\right)\wedge1;(\rho_{2k},\rho_{2k+1}]}<\infty
\]
 and Chen's relation is satisfied on $(\rho_{2k},\rho_{2k+1}]$. 
\end{prop}

\begin{proof}
This follows from an application of Proposition \ref{prop:Young integrals of variable order},
however, we must be carefull with the semi norm $\parallel\mathbb{B}_{u,v}^{(\rho_{2k},\rho_{2k+1}]}\parallel_{2\alpha\left(\cdot\right)\wedge1;\left[\rho_{2k},\rho_{2k+1}\right]},$but
we can consider 
\[
\mathbb{B}_{\rho_{2k},v}^{(\rho_{2k},\rho_{2k+1}]}=\delta\left(\int_{0}^{\cdot}B_{r}dX_{r}\right)_{\rho_{2k},v}-B_{\rho_{2k}}B_{\rho_{2k},v}
\]
\[
:=\lim_{u\rightarrow\rho_{2k}}\delta\left(\int_{0}^{\cdot}B_{r}dB_{r}\right)_{u,v}-B_{u}B_{u,v},
\]
which is of sufficient regularity $P-a.s.$ (in fact, this regularity
is $1$). Chen's relation is of course not affected by this consideration. 
\end{proof}
\begin{cor}
The iterated integral $\mathbb{B}^{|\left[\rho_{i},\rho_{i+1}\right]}$
as constructed in Lemma \ref{lem:MbM rough path} satisfy Chen's relation
and the analytic requirements such that $\left(B^{h,|\left[\rho_{i},\rho_{i+1}\right]},\mathbb{B}^{|\left[\rho_{i},\rho_{i+1}\right]}\right)\in\mathscr{C}_{2}^{h\left(\cdot\right)-}\left(\left[\rho_{i},\rho_{i+1}\right];\mathbb{R}^{d}\right)$
$P$-a.s. and it follows that there exists a lift $\mathcal{C}^{\alpha\left(\cdot\right)}\left(\left[0,T\right]\right)\ni B^{h}\mapsto\left(B^{h,},\mathbb{B}\right)\in\mathscr{C}_{2}^{h\left(\cdot\right)-}\left(\left[0,T\right];\mathbb{R}^{d}\right)$.
\end{cor}

\begin{proof}
First note that 
\[
\mathbb{B}_{s,t}^{i,j}=\delta\left(\int_{0}^{\cdot}K(t,r)B_{r}^{h}(i)dB_{r}(j)\right)_{s,t}-B_{s}^{h}(i)B_{s,t}^{h}(j),
\]
and it is easily seen that 
\[
\mathbb{B}_{s,t}^{i,j}-\mathbb{B}_{s,u}^{i,j}-\mathbb{B}_{u,t}^{i,j}=B_{s,u}^{h}(i)B_{u,t}^{h}(j).
\]
 By the above relation we know 
\[
\mathbb{B}_{0,t}-\mathbb{B}_{0,s}=\mathbb{B}_{s,t}+B_{0,s}^{h}B_{s,t}^{h}.
\]
To study the regularity of the mapping $t\mapsto\mathbb{B}_{0,t}$
we know by Lemma \ref{lem:Gagliardo Embedding} that it is sufficient
to check that for all $p>\frac{1}{h_{*}}$
\[
\int_{\left[\rho_{i},\rho_{i+1}\right]^{2}}\frac{E\left[|\mathbb{B}_{s,t}|^{2p}\right]}{|t-s|^{1+\max\left(\gamma\left(t\right),\gamma\left(s\right)\right)2p}}dsdt<\infty.
\]
But since we know that $E\left[|\mathbb{B}_{s,t}|^{2}\right]\lesssim|t-s|^{\max\left(h\left(t\right),h(s)\right)},$
and since $\mathbb{B}$ belongs to the second Wiener chaos, we have
equivalence in $L^{p}$ norms, and hence 
\[
\int_{\left[\rho_{i},\rho_{i+1}\right]^{2}}\frac{E\left[|\mathbb{B}_{s,t}|^{2p}\right]}{|t-s|^{1+\max\left(\gamma\left(t\right),\gamma\left(s\right)\right)2p}}dsdt
\]
\[
\leq C\int_{\left[\rho_{i},\rho_{i+1}\right]^{2}}\frac{E\left[|\mathbb{B}_{s,t}|^{2}\right]^{p}}{|t-s|^{1+\max\left(\gamma\left(t\right),\gamma\left(s\right)\right)2p}}dsdt
\]
\[
\leq C\int_{\left[\rho_{i},\rho_{i+1}\right]^{2}}|t-s|^{\left(2\max\left(h\left(t\right),h\left(s\right)\right)-\max\left(\gamma\left(t\right),\gamma\left(s\right)\right)\right)2p-1}dsdt
\]
and the right hand side above is finite for $\left(2\max\left(h\left(t\right),h\left(s\right)\right)-\max\left(\gamma\left(t\right),\gamma\left(s\right)\right)\right)>0$
for all $p>\frac{1}{h_{*}}$, and hence we have the inclusion $t\mapsto\mathbb{B}_{0,t}\in\mathcal{C}^{\gamma\left(\cdot\right)}$
for $\gamma$ such that for all $t\in\left[0,T\right]$ , $0<\gamma\left(t\right)<h(t)$.
It now follows from 
\end{proof}
From the above theorems it is now clear that it is possible to study
multifractional Brownian motion in a pathwise framework, as there
exists at least one way of constructing the iterated integral. Under
sufficient continuity conditions on the regularity function, it is
possible to divide its domain in accordance to when it is above or
below $\frac{1}{2},$ and in this way we have shown that we can construct
the iterated integral locally on these sub domains and then glue them
together. The existence of these iterated integrals and variable order
rough paths is the basis on which we will build the next section on;
namely an introduction to variable order controlled paths. 
\begin{rem}
We will see in the next section that these iterated integrals play
a crucial role in constructing solutions to differential equations
driven by a Hölder noise with regularity below $\frac{1}{2}.$ We
can think of this as giving the the rough integral more information
about the process when the regularity of the process gets very low.
However, we do not need these integrals when the regularity function
$\alpha\left(t\right)>\frac{1}{2}$. Therefore, the construction of
the rough path on all of $\left[0,T\right]$ when $\alpha$ varies
above and below $\frac{1}{2}$ may seem unnecessary. Indeed, at the
sub-intervals of $\left[0,T\right]$ when $\alpha\left(t\right)>\frac{1}{2}$
these iterated integrals will vanish and not affect the limit of the
Riemann sum from the sewing lemma \ref{lem:-SEWING LEMMA} anyway!
However, by constructing the rough path over all of $\left[0,T\right]$
we do not need do not need to differentiate on the construction of
the solutions of differential equations in different spaces later.
In fact, when we will consider regularity functions below $\frac{1}{2}$,
we can not only consider a fixed point argument in the variable order
Hölder space, but we must also show a fixed point argument for what
is called the Gubinelli derivative of the solution. Essentially, it
would be sufficient to only construct the rough path of a variable
order path on the subsets where the regularity function is below $\alpha$,
and then construct solutions to differential equations separately,
by the same procedure as how we constructed the rough path. However,
you will be left with a collection of solutions which no longer lie
in similar spaces (i.e Hölder spaces vs spaces of controlled paths),
and therefore are much more difficult to glue together. We therefore
recommend to construct the rough path on all of $\left[0,T\right]$
before constructing solutions to variable order rough differential
equations. 
\end{rem}

\section{\label{sec:Controlled paths and variabel order RDE}Controlled variable
order paths and variable order RDE's. }

In this section we will generalize some existence and uniqueness results
of RDE's of the form 
\begin{equation}
dY=f\left(Y\right)d\mathbf{X},\,\,\,Y_{0}=y\label{eq:Variable order RDE}
\end{equation}
 where $\mathbf{X}\in\mathscr{C}_{2}^{\alpha\left(\cdot\right)}\left(\left[0,T\right];\mathbb{R}^{d}\right)$
with $\alpha:\left[0,T\right]\rightarrow\left(0,1\right)$ and $f:\mathbb{R}^{n}\rightarrow\mathbb{R}^{d\times n}$.
We will study these equations through the familiar theory of controlled
paths established in \cite{Gubinelli} and also later used in \cite{FriHai},
but generalize the spaces to account for the variable regularity of
the paths. 

Let us first define the space of variable order controlled path. 
\begin{defn}
\label{def:Controlled variable order path }Let $X\in\mathcal{C}^{\alpha\left(\cdot\right)}\left(\left[0,T\right];\mathbb{R}^{d}\right)$
with $\alpha:\left[0,T\right]\rightarrow\left(\frac{1}{3},1\right)$,
If $Y:\left[0,T\right]\rightarrow\mathbb{R}^{n}$ and $Y':\left[0,T\right]\rightarrow\mathbb{R}^{d\times n}$satisfy
the relation 
\[
Y_{s,t}=Y'_{s}X_{s,t}+R_{s,t}^{Y},
\]
where $|R_{s,t}^{Y}|\lesssim|t-s|^{2\alpha\left(s\right)}$ and $|Y'_{s,t}|\lesssim|t-s|^{\alpha\left(s\right)}$,
then we call the tuple $\left(Y,Y'\right)\in\mathcal{D}_{X}^{\alpha\left(\cdot\right)}\left(\left[0,T\right];\mathbb{R}^{n}\right)$
a (variable order) controlled path with respect to $X$ , and we define
the following semi-norm on $\mathcal{D}_{X}^{\alpha\left(\cdot\right)}$
\[
\parallel Y,Y'\parallel_{X,\alpha\left(\cdot\right)}:=\parallel Y'\parallel_{\alpha\left(\cdot\right)}+\parallel R^{Y}\parallel_{2\alpha\left(\cdot\right)}.
\]
Furthermore, the mapping $\left(Y,Y'\right)\mapsto|Y_{0}|+|Y'_{0}|+\parallel Y,Y'\parallel_{X,\alpha\left(\cdot\right)}$
induces another norm on $\mathcal{D}_{X}^{\alpha\left(\cdot\right)}$
which makes it a linear Banach space, and in fact is complete under
the metric induced by this norm. 
\end{defn}

Notice also that the Hölder regularity of the first order component
in a controlled path $\left(Y,Y'\right)$ is implied by the inequality
\begin{equation}
\parallel Y\parallel_{\alpha\left(\cdot\right)}\leq\parallel Y'\parallel_{\infty}\parallel X\parallel_{\alpha\left(\cdot\right)}+\parallel R^{Y}\parallel_{2\alpha\left(\cdot\right)}T^{\alpha},\label{eq:regularity of first component in a controlled path}
\end{equation}
obtained by the assumed relation in Definition \ref{def:Controlled variable order path }. 

It is quite straight forward to construct a integral of a controlled
path with respect to the path $X$. Using the sewing lemma, we present
the following proposition for variable order rough integration. 
\begin{prop}
\label{prop:Variable order rough integral }Let $\left(X,\mathbb{X}\right)\in\mathscr{C}^{\alpha\left(\cdot\right)}\left(\left[0,T\right];\mathbb{R}^{d}\right)$
and $\left(Y,Y'\right)\in\mathcal{D}_{X}^{\alpha\left(\cdot\right)}\left(\left[0,T\right];\mathbb{R}^{n}\right)$
for some $\alpha:\left[0,T\right]\rightarrow\left(\frac{1}{3},1\right)$,
then 
\[
|\int_{s}^{t}Y_{r}dX_{r}-Y_{s}X_{s,t}-Y'_{s}\mathbb{X}_{s,t}|
\]
\[
\lesssim|t-s|^{\beta}\left(\parallel X\parallel_{\alpha\left(\cdot\right)}+\parallel\mathbb{X}\parallel_{2\alpha\left(\cdot\right)\wedge1}\right)\left(\parallel Y,Y'\parallel_{X,\alpha\left(\cdot\right)}\right)
\]
\end{prop}

\begin{proof}
Set $\varXi_{s,t}=Y_{s}X_{s,t}+Y'_{s}\mathbb{X}_{s,t},$ and see that
by using the Chen relation, we obtain 
\[
\delta\varXi_{s,u,t}=Y'_{s,u}\mathbb{X}_{u,t}+R_{s,u}^{Y}X_{u,t}.
\]
The result now follows from the Sewing lemma \ref{lem:-SEWING LEMMA},
and using the Hölder regularity estimates similar to that obtained
in proposition \ref{prop:Young integrals of variable order}. 
\end{proof}
Let us also give lemma showing that functions of controlled paths
are still controlled path. 
\begin{lem}
\label{lem:Controll of functions of paths}Let $f\in C_{b}^{2}\left(\mathbb{R}^{n},\mathbb{R}^{n}\right)$
and $\left(Y,Y\right)\in\mathcal{D}_{X}^{\alpha\left(\cdot\right)}\left(\left[0,T\right];\mathbb{R}^{n}\right)$
with $X\in\mathcal{C}^{\alpha\left(\cdot\right)}\left(\left[0,T\right];\mathbb{R}^{d}\right)$
and $\alpha:\left[0,T\right]\rightarrow\left(\frac{1}{3},1\right).$
Then $\left(f\left(Y\right),f'\left(Y\right)Y'\right)\in\mathcal{D}_{X}^{\alpha\left(\cdot\right)}\left(\left[0,T\right];\mathbb{R}^{n}\right),$
and there exists a $C>0$ such that 
\[
\parallel f\left(Y\right),f'\left(Y\right)Y'\parallel_{X,\alpha\left(\cdot\right)}\leq C\left(|Y_{0}|+|Y'_{0}|+\parallel Y,Y'\parallel_{X,\alpha\left(\cdot\right)}\right).
\]
\end{lem}

\begin{proof}
We must show two components; first show that $R_{s,t}^{f\left(Y\right)}$
defined as 
\[
f\left(Y\right)_{s,t}-f'\left(Y_{s}\right)Y'_{s}X_{s,t}=R_{s,t}^{f\left(Y\right)}
\]
 is of $2\alpha\left(\cdot\right)$ regularity, and then we will show
that $f'\left(Y\right)Y'$ is of $\alpha\left(\cdot\right)$ regularity.
Let us start with $R^{f\left(Y\right)}.$ Note that since $\left(Y,Y'\right)\in\mathcal{D}_{X}^{\alpha\left(\cdot\right)}$
we have $Y_{s,t}=Y'_{s}X_{s,t}+R_{s,t}^{Y}$ therefore 
\[
R_{s,t}^{f\left(Y\right)}=f\left(Y\right)_{s,t}-f'\left(Y_{s}\right)Y_{s,t}+f'\left(Y_{s}\right)R_{s,t}^{Y},
\]
and it is clear that 
\[
|f'\left(Y_{s}\right)R_{s,t}^{Y}|\leq\parallel f\parallel_{C_{b}^{1}}\parallel R^{Y}\parallel_{2\alpha\left(\cdot\right)}|t-s|^{2\alpha\left(s\right)}.
\]
 Moreover, note that by standard Taylor approximation of $f\left(Y_{t}\right)$
around $Y_{s}$ we have that 
\[
|f\left(Y\right)_{s,t}-f'\left(Y_{s}\right)Y_{s,t}|\lesssim\parallel f\parallel_{C_{b}^{2}}|Y_{s,t}|^{2}\lesssim\parallel Y,Y'\parallel_{X,\alpha\left(\cdot\right)}|t-s|^{2\alpha\left(s\right)},
\]
 which concludes that $|R_{s,t}^{f\left(Y\right)}|\leq C\parallel Y,Y'\parallel_{X,\alpha\left(\cdot\right)}|t-s|^{2\alpha\left(s\right)}.$
For the second part, we use that $|a\tilde{a}-b\tilde{b}|\lesssim|a||\tilde{a}-\tilde{b}|+|a-b||\tilde{b}|$
and get that 
\[
|f'\left(Y_{t}\right)Y'_{t}-f'\left(Y_{s}\right)Y'_{s}|\leq|f'\left(Y_{t}\right)||Y'_{t}-Y'_{s}|+|f'\left(Y_{t}\right)-f'\left(Y_{s}\right)||Y'_{s}|.
\]
 And it follows that 
\[
\parallel f'\left(Y\right)Y'\parallel_{\alpha\left(\cdot\right)}\leq C\left(\parallel f\parallel_{C_{b}^{2}}\right)\left(|Y_{0}|+|Y'_{0}|+\parallel Y,Y'\parallel_{X,\alpha\left(\cdot\right)}\right).
\]
\end{proof}
We will now show existence and uniqueness of solutions to equation
(\ref{eq:Variable order RDE}) by a fixed point argument in the space
of controlled paths $\mathcal{D}_{X}^{\alpha\left(\cdot\right)}$.
\begin{thm}
Let $\mathbf{X}=\left(X,\mathbb{X}\right)\in\mathscr{C}^{\alpha\left(\cdot\right)}\left(\left[0,T\right];\mathbb{R}^{d}\right)$
and let $f\in C_{b}^{3}\left(\mathbb{R}^{n};\mathbb{R}^{d\times n}\right)$,
then there exists a unique solution $\left(Y,Y'\right)\in\mathcal{D}_{X}^{\alpha\left(\cdot\right)}\left(\left[0,T\right];\mathbb{R}^{n}\right)$
to equation 
\[
dY_{t}=f\left(Y_{t}\right)d\mathbf{X}_{t},\,\,\,Y_{0}=y\in\mathbb{R}^{n}.
\]
 
\end{thm}

\begin{proof}
We will only present a sketch of proof, as the proof is a simple generalization
of Theorem 8.4 in \cite{FriHai}. We will therefore advise the interested
reader to go through that proof first. Note that the solution is constructed
in $\mathcal{D}_{X}^{\alpha\left(\cdot\right)-\epsilon}\left(\left[0,T\right];\mathbb{R}^{n}\right)$,
as this makes the argument considerably simpler by using that for
$X\in\mathcal{C}^{\alpha\left(\cdot\right)}$ we have $\parallel X\parallel_{\alpha-\epsilon}\leq T^{\epsilon}\parallel X\parallel_{\alpha\left(\cdot\right)}$
and then shown to actually be included in $\mathcal{D}_{X}^{\alpha\left(\cdot\right)}$
which follows from the fact that the regularity of $Y$ is inherited
from the regularity of $X$. We construct the solution by a fixed
point argument in both the solution $Y$ itself, and $Y'$. Assume
$\left(Y,Y'\right)\in\mathcal{D}_{X}^{\alpha\left(\cdot\right)}$
and consider the ``solution operator'' $\mathcal{V}_{T}:\mathcal{D}_{X}^{\alpha\left(\cdot\right)-\epsilon}\rightarrow\mathcal{D}_{X}^{\alpha\left(\cdot\right)-\epsilon}$
defined by the mapping 
\[
\left(Y,Y'\right)\mapsto\mathcal{V}_{T}\left(Y,Y'\right):=\left\{ \left(y+\int_{0}^{t}f\left(Y_{r}\right)d\mathbf{X}_{r},f\left(Y_{r}\right)\right);t\in\left[0,T\right]\right\} .
\]
From proposition \ref{prop:Variable order rough integral } and lemma
\ref{lem:Controll of functions of paths} it is clear that $\mathcal{V}_{T}\left(Y,Y'\right)\in\mathcal{D}_{X}^{\alpha\left(\cdot\right)-\epsilon}.$
One then need to show invariance of a unit ball $\mathcal{B}_{T}\subset\mathcal{D}_{X}^{\alpha\left(\cdot\right)-\epsilon}$
centered at the trivial element $t\mapsto\left(\zeta_{t},\zeta'_{t}\right)=\left(y+f\left(y\right)X_{0,t},f\left(y\right)\right)\in\mathcal{D}_{X}^{\alpha\left(\cdot\right)-\epsilon}$
under the mapping $\mathcal{V}_{T}$. That is, we will first show
$\mathcal{V}_{T}\left(\mathcal{B}_{T}\right)\subset\mathcal{B}_{T}$,
which imply that we must show 
\[
\parallel\mathcal{V}_{T}\left(Y,Y'\right)-\left(\zeta,\zeta'\right)\parallel_{\alpha\left(\cdot\right)-\epsilon,X;\left[0,T\right]}\leq1.
\]
Actually, by the triangle inequality and the fact that the Hölder
norm of $\zeta'$ are zero, it is sufficient to show that 
\[
\parallel\mathcal{V}_{T}\left(Y,Y'\right)\parallel_{\alpha\left(\cdot\right)-\epsilon,X;\left[0,T\right]}\leq1.
\]
 This follows from a combination of proposition \ref{prop:Variable order rough integral }
and lemma \ref{lem:Controll of functions of paths} which tells us
\[
\parallel R^{\mathcal{V}_{T}\left(Y,Y'\right)}\parallel_{2\left(\alpha\left(\cdot\right)-\epsilon\right)}\leq C\left(\parallel X\parallel_{\alpha\left(\cdot\right)}+\parallel\mathbb{X}\parallel_{2\alpha\left(\cdot\right)\wedge1}\right)
\]
\[
\times\left(|Y_{0}|+|Y'_{0}|+\parallel Y,Y'\parallel_{X,\alpha\left(\cdot\right)-\epsilon}\right)T^{\epsilon},
\]
and using the above together with the estimate from equation \ref{eq:regularity of first component in a controlled path},
we have 
\[
\parallel f\left(Y\right)\parallel_{\alpha\left(\cdot\right)-\epsilon}\leq C\left(|Y_{0}|+|Y'_{0}|+\parallel Y,Y'\parallel_{X,\alpha\left(\cdot\right)-\epsilon}\right)\parallel X\parallel_{\alpha\left(\cdot\right)}T^{\epsilon}.
\]
Combining the above, gives us for some $C>0$ depending on $\left(X,\mathbb{X}\right)$
, $\alpha$, and $T$, 
\[
\parallel\mathcal{V}_{T}\left(Y,Y'\right)\parallel_{\alpha\left(\cdot\right)-\epsilon,X;\left[0,T\right]}\leq C\left(|Y_{0}|+|Y'_{0}|+\parallel Y,Y'\parallel_{X,\alpha\left(\cdot\right)-\epsilon}\right)T^{\epsilon}.
\]
It now follows that we can choose $T=T_{0}$ small enough such that
\[
\parallel\mathcal{V}_{T_{0}}\left(Y,Y'\right)\parallel_{\alpha\left(\cdot\right)-\epsilon,X;\left[0,T_{0}\right]}\leq1.
\]
 The next step is to show that $\mathcal{V}_{T}$ is a contraction
on $\mathcal{D}_{X}^{\alpha\left(\cdot\right)-\epsilon}$. Assume
that $Y,\tilde{Y}\in\mathcal{C}_{X}^{\alpha\left(\cdot\right)-\epsilon}$
both starting in $y\in\mathbb{R}^{d}$, and write $\triangle_{t}=f\left(Y\right)_{t}-f\left(\tilde{Y}\right)_{t}$
, and $\triangle_{t}'=f\left(Y\right)'_{t}-f\left(\tilde{Y}\right)'_{t}$
. We need to prove that there exists a $q\in\left(0,1\right)$ such
that 
\[
\parallel\int_{0}^{\cdot}\triangle_{r}dX_{r},\triangle\parallel_{\alpha\left(\cdot\right)-\epsilon,X;\left[0,T\right]}\leq q\parallel Y-\tilde{Y},Y'-\tilde{Y}'\parallel_{\alpha\left(\cdot\right)-\epsilon,X;\left[0,T\right]}.
\]
By Proposition \ref{prop:Variable order rough integral }, we know
that 
\[
\parallel\int_{0}^{\cdot}\triangle_{r}dX_{r},\triangle\parallel_{\alpha\left(\cdot\right)-\epsilon,X,\left[0,T\right]}\leq\parallel\triangle\parallel_{\alpha\left(\cdot\right)-\epsilon}
\]
\[
+C\left(\parallel X\parallel_{\alpha\left(\cdot\right)}+\parallel\mathbb{X}\parallel_{2\alpha\left(\cdot\right)\wedge1}\right)\parallel\triangle,\triangle'\parallel_{X,\alpha\left(\cdot\right)-\epsilon}T^{\epsilon}
\]
\[
\leq C\parallel f\parallel_{C_{b}^{2}}\parallel Y-\tilde{Y}\parallel_{\alpha\left(\cdot\right)-\epsilon}+C\parallel\triangle,\triangle'\parallel_{X,\alpha\left(\cdot\right)-\epsilon}T^{\epsilon}.
\]
But from equation \ref{eq:regularity of first component in a controlled path}
we know that 
\[
\parallel Y-\tilde{Y}\parallel_{\alpha\left(\cdot\right)-\epsilon}\leq C\parallel Y-\tilde{Y},Y'-\tilde{Y}'\parallel_{\alpha\left(\cdot\right)-\epsilon,X;\left[0,T\right]}T^{\epsilon}.
\]
 It therefore remains to prove that 
\[
\parallel\triangle,\triangle'\parallel_{X,\alpha\left(\cdot\right)-\epsilon}\leq C\parallel Y-\tilde{Y},Y'-\tilde{Y}'\parallel_{\alpha\left(\cdot\right)-\epsilon,X;\left[0,T\right]}.
\]
 Using the fact that $f$ is differentiable, we can write 
\[
f\left(x\right)-f\left(y\right)=g\left(x,y\right)\left(x-y\right),\,\,\,g\left(x,y\right)=\int_{0}^{1}Df\left(x\rho+\left(1-\rho\right)y\right)d\rho,
\]
 we have that 
\[
\triangle_{s}=g\left(Y_{s},\tilde{Y}_{s}\right)h\left(Y_{s},\tilde{Y}_{s}\right),
\]
where $h\left(x,y\right)=x-y$. For simplicity, write $g_{s}=g\left(Y_{s},\tilde{Y}_{s}\right)$,
and similarly for $h$. Since $f\in C_{b}^{3}$ we have that $g\in C_{b}^{2},$
and using Lemma \ref{lem:Controll of functions of paths} it follows
that 
\[
\parallel g,g'\parallel_{\alpha\left(\cdot\right)-\epsilon,X,\left[0,T\right]}<\parallel f\parallel_{C_{b}^{3}}.
\]
Moreover, it is shown in \cite{FriHai} in the proof of Theorem 8.4
(and it is a simple exercise to check) that $\mathcal{D}_{X}^{\alpha\left(\cdot\right)}$
is a Banach algebra, and in fact 
\[
\parallel gh,\left(gh\right)'\parallel_{\alpha\left(\cdot\right)-\epsilon,X,\left[0,T\right]}\leq\left(|g_{0}|+|g'_{0}|+\parallel g,g'\parallel_{\alpha\left(\cdot\right)-\epsilon,X,\left[0,T\right]}\right)
\]
\[
\times\left(|h_{0}|+|h'_{0}|+\parallel h,h'\parallel_{\alpha\left(\cdot\right)-\epsilon,X,\left[0,T\right]}\right).
\]
 Using that $h_{0}=h'_{0}=0,$ and the the previous estimates for
$\left(g,g'\right)$, we can see that 
\[
\parallel\triangle,\triangle'\parallel_{X,\alpha\left(\cdot\right)-\epsilon}\leq\parallel f\parallel_{C_{b}^{3}}\parallel Y-\tilde{Y},Y'-\tilde{Y}'\parallel_{\alpha\left(\cdot\right)-\epsilon,X;\left[0,T\right]}.
\]
 It is now clear that there exists a constant $C>0$ depending on
$X,\mathbb{X}$, $\alpha$,$\parallel f\parallel_{C_{b}^{3}}$, and
(uniformly) on $T$, such that 
\[
\parallel\int_{0}^{\cdot}\triangle_{r}dX_{r},\triangle\parallel_{\alpha\left(\cdot\right)-\epsilon,X;\left[0,T\right]}\leq C\parallel Y-\tilde{Y},Y'-\tilde{Y}'\parallel_{\alpha\left(\cdot\right)-\epsilon,X;\left[0,T\right]}T^{\epsilon},
\]
again, choose $T=\tilde{T}_{0}$ small enough such that $C\tilde{T}_{0}^{\epsilon}<1$.
Choose $T_{0}=T_{0}\wedge\tilde{T}_{0}$ and the existence of a unique
solution follows from the Banach fixed point theorem. Now, we must
iterate the procedure on the proceeding intervals $\left[T_{0},2T_{0}\right],..,\left[kT_{0},T\right]$
for some $k\in\mathbb{N}$ by considering the same equation but with
initial value being given by the solution on the previous interval.
Then using Lemma \ref{lem: Scaling of H=0000F6lder norms}, a unique
solution on the whole interval $\left[0,T\right]$ follows. 
\end{proof}
\begin{rem}
The use of variable order rough paths to consider delay equations
where attempted during the work on this paper, but proved to be more
difficult than in regular rough paths theory. The reason is that when
considering the fixed point argument of a variable order RDE, we must
consider the regularities both of the solution and its Gubinelli derivative.
However, the Gubinelli derivative of $\int_{0}^{\cdot}f\left(Y_{r-\tau}\right)dX_{r}$
is given by $f\left(Y_{r-\tau}\right)$ which has Hölder regularity
according to $\alpha\left(\cdot-\tau\right),$but the integral $\int_{0}^{\cdot}f\left(Y_{r-\tau}\right)dX_{r}$
inherits the $\alpha\left(\cdot\right)$ regularity of the driving
noise $X$. Thus the the topology relating to the space $\mathcal{C}_{X}^{\alpha\left(\cdot\right)}$
needed to be changed to consider this problem. We would therefore
like to study this in more detail at a later time, and for now only
cover regular variable order rough path theory, which we hope will
be useful in later research on variable order paths. 
\end{rem}

\section{Appendix}

We will here prove some technical results needed in the article. First,
denote by $B_{r}\left(x\right)=\left\{ y\in\mathbb{R}||y|\leq r\right\} ,$
and we use that for an integrabl function $f$, we write 
\[
\langle f\rangle_{\left[0,T\right]}=\frac{1}{T}\int_{0}^{T}f(y)dy.
\]
 
\begin{prop}
\label{prop:Lebesgues Trick...}Assume $p\geq1$ and $\gamma:\left[0,T\right]\rightarrow(1,p],$
and assume $f:\left[0,T\right]\rightarrow\mathbb{R}$ satisfies for
all $t\in\left[0,T\right]$
\[
\left[f\right]_{p,p\gamma\left(t\right);\left[0,T\right]}:=\left(\sup_{\epsilon>0,t\in\left[0,T-\epsilon\right]}\left\{ \epsilon^{-\gamma\left(t\right)p}\int_{B_{\frac{\epsilon}{2}}\left(t\right)}|f(x)-\langle f\rangle_{B_{\frac{\epsilon}{2}}\left(t\right)}|^{p}dx\right\} \right)^{\frac{1}{p}}<\infty,
\]
then for $0\leq R\leq R'$
\[
|\langle f\rangle_{B_{R}\left(t\right)}-\langle f\rangle_{B_{R'}\left(t\right)}|\leq\left[f\right]_{p,p\gamma\left(t\right);\left[0,T\right]}|B_{R'}\left(t\right)|^{\gamma\left(t\right)-\frac{1}{p}}.
\]
 
\end{prop}

\begin{proof}
We know that $B_{R}\left(t\right)\subset B_{R'}\left(t\right)$ and
we have by the Hölder inequality 
\[
|\frac{1}{R'}\int_{B_{R'}\left(t\right)}f(x)-\langle f\rangle_{B_{R}\left(t\right)}dx|^{p}\leq\frac{R'^{\frac{p-1}{p}}}{R'}\left(\int_{B_{R'}\left(t\right)}|f(x)-\langle f\rangle_{B_{R}\left(t\right)}|^{p}dx\right)^{\frac{1}{p}}
\]
 and since $\gamma:\left[0,T\right]\rightarrow(1,p],$ we have 
\[
\leq R'^{\gamma\left(t\right)-\frac{1}{p}}\left(\sup_{\epsilon>}\epsilon^{-p\gamma\left(t\right)}\int_{B_{R'}\left(t\right)}|f(x)-\langle f\rangle_{B_{\epsilon}\left(t\right)}|^{p}dx\right)^{\frac{1}{p}},
\]
which completes our proof. 
\end{proof}
\begin{lem}
\label{lem:Gagliardo Embedding}Let $f:\left[0,T\right]\rightarrow\mathbb{R},$
be continuous and assume $\alpha\left(t\right)>\frac{1}{p}$ for all
$t\in\left[0,T\right]$. Then 
\[
\parallel f\parallel_{\alpha\left(\cdot\right);\left[0,T\right]}\leq\int_{0}^{T}\int_{0}^{T}\frac{|f(x)-f(y)|^{p}}{|x-y|^{1+\max\left(\gamma\left(x\right),\gamma\left(y\right)\right)p}}dxdt,
\]
where $\gamma\left(t\right)=\alpha\left(t\right)-\frac{1}{p}$. 
\end{lem}

\begin{proof}
Consider the average of a function, i.e 
\[
\langle f\rangle_{\left[0,T\right]}=\frac{1}{T}\int_{0}^{T}f(y)dy,
\]
then for some $\xi\in\mathbb{R}$ , we have by Hölder's inequality
that 
\[
|\xi-\langle f\rangle_{\left[0,T\right]}|^{p}=\frac{1}{T^{p}}|\int_{0}^{T}\xi-f\left(y\right)dy|^{p}
\]
\[
\leq\frac{1}{T}\int_{0}^{T}|\xi-f\left(y\right)|^{p}dy.
\]
Consider now $\xi=\langle f\rangle_{B_{R}\left(t\right)}$ for some
$R>0$ and then 
\[
\int_{B_{R}\left(t\right)}|f(x)-\langle f\rangle_{B_{R}\left(t\right)}|^{p}dx\leq\frac{1}{R}\int_{B_{R}\left(t\right)}\int_{B_{R}\left(t\right)}|f(x)-f(y)|^{p}dxdy.
\]
We know that $|t-s|\leq R$ we have that for some $\gamma:\left[0,T\right]\rightarrow\left(\gamma_{*},\gamma^{*}\right)$
\[
\int_{B_{R}\left(t\right)}|f(x)-\langle f\rangle_{B_{R}\left(t\right)}|^{p}dx
\]
\[
\leq\frac{1}{R}\int_{B_{R}\left(t\right)}\int_{B_{R}\left(t\right)}\left(|x-y|^{1+\max\left(\gamma\left(x\right),\gamma\left(y\right)\right)p}\right)\frac{|f(x)-f(y)|^{p}}{|x-y|^{1+\max\left(\gamma\left(x\right),\gamma\left(y\right)\right)p}}dxdy
\]
 
\[
\leq R^{\tilde{\gamma}\left(t,s\right)p}\int_{B_{R}\left(t\right)}\int_{B_{R}\left(t\right)}\frac{|f(x)-f(y)|^{p}}{|x-y|^{1+\max\left(\gamma\left(x\right),\gamma\left(y\right)\right)p}}dxdy
\]
 where $\gamma\left(t,\epsilon\right)=\max_{u\in\left[s,t\right]}\left\{ \alpha\left(u\right)\right\} $.
From this, we can conclude that 
\[
[f]_{p,\gamma\left(\cdot\right);\left[0,T\right]}^{p}:=\sup_{R>0,t\in\left[0,T-\epsilon\right]}\left\{ R^{-\gamma\left(t\right)p}\int_{B_{R}\left(t\right)}|f(x)-\langle f\rangle_{B_{R}\left(t\right)}|^{p}dx\right\} 
\]
\[
\leq\int_{0}^{T}\int_{0}^{T}\frac{|f(x)-f(y)|^{p}}{|x-y|^{1+\max\left(\gamma\left(x\right),\gamma\left(y\right)\right)p}}dxdy.
\]
Moreover, for $t,s\in\left[0,T\right]$ with $|t-s|=R$ we have 
\[
|f\left(t\right)-f(s)|\leq|f(t)-\langle f\rangle_{B_{2R}\left(t\right)}|
\]
\[
+|\langle f\rangle_{B_{2R}\left(t\right)}-\langle f\rangle_{B_{2R}\left(s\right)}|+|\langle f\rangle_{B_{2R}\left(s\right)}-f(s)|.
\]
By the continuity of $f$ we have by Lebesgue's differentiation theorem
and Proposition \ref{prop:Lebesgues Trick...}, we have that 
\[
|f(t)-\langle f\rangle_{B_{2R}\left(t\right)}|\leq\left[f\right]_{p,p\gamma\left(\cdot\right)}|B_{2R}\left(t\right)|^{\gamma\left(t\right)-\frac{1}{p}},
\]
and similarly for the third term 
\[
|\langle f\rangle_{B_{2R}\left(s\right)}-f(s)|\leq\left[f\right]_{p,p\gamma\left(\cdot\right)}|B_{2R}\left(s\right)|^{\gamma\left(s\right)-\frac{1}{p}}.
\]
For the middle term, we can see that for some $f(z)$ we have 
\[
|\langle f\rangle_{B_{2R}\left(t\right)}-\langle f\rangle_{B_{2R}\left(s\right)}|\leq|f(z)-\langle f\rangle_{B_{2R}\left(t\right)}|+|f(z)-\langle f\rangle_{B_{2R}\left(s\right)}|,
\]
 and we can choose $z\in B_{2R}\left(t\right)\cap B_{2R}\left(s\right)$
and integrate w.r.t. $z$ and find 
\[
|B_{2R}\left(t\right)\cap B_{2R}\left(s\right)||\langle f\rangle_{B_{2R}\left(t\right)}-\langle f\rangle_{B_{2R}\left(s\right)}|
\]
\[
\leq\int_{B_{2R}\left(t\right)\cap B_{2R}\left(s\right)}|f(z)-\langle f\rangle_{B_{2R}\left(t\right)}|dz+\int_{B_{2R}\left(t\right)\cap B_{2R}\left(s\right)}|f(z)-\langle f\rangle_{B_{2R}\left(s\right)}|dz
\]
 
\[
\leq\int_{B_{2R}\left(t\right)}|f(z)-\langle f\rangle_{B_{2R}\left(t\right)}|dz+\int_{B_{2R}\left(s\right)}|f(z)-\langle f\rangle_{B_{2R}\left(s\right)}|dz.
\]
Also, it is clear that since $B_{R}\left(t\right)\cup B_{R}\left(s\right)\subset B_{2R}\left(t\right)\cap B_{2R}\left(s\right)$
and hence 
\[
|B_{R}\left(t\right)|\leq|B_{2R}\left(t\right)\cap B_{2R}\left(s\right)|,
\]
 and the same for $B_{R}\left(s\right)$ and we obtain 
\[
|\langle f\rangle_{B_{2R}\left(t\right)}-\langle f\rangle_{B_{2R}\left(s\right)}|\leq\frac{1}{|B_{R}\left(t\right)|}\int_{B_{2R}\left(t\right)}|f(z)-\langle f\rangle_{B_{2R}\left(t\right)}|dz
\]
\[
+\frac{1}{|B_{R}\left(s\right)|}\int_{B_{2R}\left(s\right)}|f(z)-\langle f\rangle_{B_{2R}\left(s\right)}|dz.
\]
 Using Hölders inequality, it is easily obtained that 
\[
\frac{1}{|B_{R}\left(t\right)|}\int_{B_{2R}\left(t\right)}|f(z)-\langle f\rangle_{B_{2R}\left(t\right)}|dz\leq cR^{\gamma\left(t\right)-\frac{1}{p}}\left[f\right]_{p,\gamma\left(\cdot\right)p},
\]
and similarly for the integral around of the ball around $s$. Combining
the above inequalities we get 
\[
|f(t)-f(s)|\leq2\left[f\right]_{p,p\gamma\left(\cdot\right)}\left(|t-s|^{\gamma\left(t\right)-\frac{1}{p}}+|t-s|^{\gamma\left(s\right)-\frac{1}{p}}\right)
\]
Using the assumption on $\gamma$ we have that $\sup_{t,s\in\left[0,T\right]}|t-s|^{\gamma\left(t\right)-\gamma\left(s\right)}\leq c_{T}$
and set $\alpha\left(t\right)=\gamma\left(t\right)-\frac{1}{p},$
we have 
\[
\sup_{t\neq s\in\left[0,T\right]}\frac{|f(t)-f(s)|}{|t-s|^{\max\left(\alpha\left(t\right),\alpha\left(s\right)\right)}}\leq c_{T,\alpha}\left[f\right]_{p,p\gamma\left(\cdot\right)},
\]
which leads to the conclusion 
\[
\parallel f\parallel_{\alpha\left(\cdot\right);\left[0,T\right]}\leq\int_{0}^{T}\int_{0}^{T}\frac{|f(x)-f(y)|^{p}}{|x-y|^{1+\max\left(\gamma\left(x\right),\gamma\left(y\right)\right)p}}dxdy.
\]
\end{proof}
 
\end{document}